\documentclass[12pt]{article}

\usepackage{inputenc}%[latin1]
\usepackage{amsmath}
\usepackage{amsfonts}
\usepackage{amssymb}
\usepackage{graphicx}
\usepackage[english]{babel}
\usepackage{mathrsfs}
\usepackage{amsthm}
\usepackage[all]{xy}
\usepackage{geometry}
\usepackage{xcolor}

\allowdisplaybreaks
\title{Propagation in a fractional reaction-diffusion equation in a periodically hostile environment}
\author{Alexis L\'{e}culier\footnote{Institut de Math\'{e}matiques de Toulouse; UMR 5219, Universit\'{e} de Toulouse; CNRS, UPS IMT, F-31062 Toulouse Cedex 9, France; E-mail:  Alexis.Leculier@math.univ-toulouse.fr} , Sepideh Mirrahimi\footnote{Institut de Math\'{e}matiques de Toulouse; UMR 5219, Universit\'{e} de Toulouse; CNRS, UPS IMT, F-31062 Toulouse Cedex 9, France; E-mail:  Sepideh.Mirrahimi@math.univ-toulouse.fr}  and Jean-Michel Roquejoffre\footnote{Institut de Math\'{e}matiques de Toulouse; UMR 5219, Universit\'{e} de Toulouse; CNRS, UPS IMT, F-31062 Toulouse Cedex 9, France; E-mail:  Jean-Michel.Roquejoffre@math.univ-toulouse.fr} }
\begin{document}                     
\maketitle

\begin{abstract}
\noindent We provide an asymptotic analysis of a fractional Fisher-KPP type equation in periodic non-connected media with Dirichlet conditions outside the domain. After showing the existence and uniqueness of a non-trivial bounded stationary state $n_+$, we prove that it invades the unstable state zero exponentially fast in time.
\end{abstract}
\noindent{\bf Key-Words: } Non-local fractional operator, Fisher KPP, asymptotic analysis, exponential speed of propagation, perturbed test function\\
\noindent{\bf AMS Class. No:} {35K08, 35K57, 35B40, 35Q92.}

\tableofcontents

\newtheorem{theorem}{Theorem}
\newtheorem*{theorem*}{Theorem}
\newtheorem{corollary}{Corollary}
\newtheorem{lemma}{Lemma}
\newtheorem*{lemma*}{Lemma}
\newtheorem{definition}{Definition}
\newtheorem{proposition}{Proposition}
\newtheorem{definition-proposition}{Definition-Proposition}
\newtheorem*{notation}{Notation}
\newtheorem*{remark}{Remark}

\section{Introduction}

\subsection{Model and question}

We focus on the following equation :
\begin{equation}
\label{equation_1}
  \left\{
      \begin{aligned}
       &\partial_t n(x,t) + (-\Delta)^{\alpha} n(x,t)=n(x,t)(1-n(x,t))&& \ \text{ for } (x,t) \in \Omega \times ]0,\infty[,\\
       &n(x,t)=0&& \ \text{ for } (x,t) \in \Omega^c \times [0,\infty[,\\
       &n(x,0)=n_0(x),\\       
      \end{aligned}
    \right. 
\end{equation}
\textcolor{black}{\textcolor{black}{where} $\Omega$ is a periodic domain of $\mathbb{R}^d$ that will be specified later on, $n_0$ a compactly supported initial data and $(-\Delta)^\alpha$ the fractional Laplacian with $\alpha \in ]0,1[$ which is defined as follows :
\[\forall (x,t) \in \mathbb{R}^d \times ]0, +\infty[, \quad (-\Delta)^\alpha n (x,t)=C_\alpha \ PV \int_{\mathbb{R}^d} \dfrac{n(x,t)-n(y,t)}{|x-y|^{d+2\alpha}}dy \ \text{ where } C_\alpha = \dfrac{4^\alpha \Gamma(\frac{d}{2}+\alpha)}{\pi^{\frac{d}{2}}|\Gamma(-\alpha)|}.\]}
The main aim of this paper is to describe the propagation front associated to \eqref{equation_1}. We show that the stable state invades the unstable state with an exponential speed. 

%\subsection{Motivation}

Equation \eqref{equation_1} models the dynamic of a species subject to a non-local dispersion in a periodically hostile environment. The quantity $n(x,t)$ stands for the density of the population at position $x$ and time $t$. The fractional Laplacian describes the motion of individuals, it takes into account the possibility of "\textcolor{black}{large} jump" (move rapidly) of individuals from one point to another \textcolor{black}{with a high rate}, for instance because of human activities for animals or because of the wind for seeds. %\textbf{\textcolor{black}{Description plus precise requise.}} 
The term $(1-n(x,t))$ represents the growth rate of the population at position $x$ and time $t$. The originality of this model is the following, \textcolor{black}{the reachable areas for the species are disconnected and periodic.} Here, we assume that the regions where the species can develop itself are homogeneous.% Thanks to the non-local diffusion (which models the "jumps"), the species will invade all the "good" patches and the solution will converge to a non-null stable stationary state with a speed which growths exponentially fast. 

Many works deal with the case of a standard diffusion ($\alpha = 1$, see \cite{Hitchhiker} for a proof of the passage from the non-local to the local character of $(-\Delta)^\alpha$) with homogenous or heterogeneous environment (see \cite{Fisher}, \cite{KPP1937}, \cite{Aronson-Weinberger} and \cite{Freidlin_Gertner}). Closer to this article, Guo and Hamel in \cite{Guo-Hamel} focus on a Fisher-KPP equation with periodically hostile regions and a standard diffusion. The authors prove that the stable state invades the unstable state in the connected component of the support of the initial data. In our work, thanks to the non-local character of the fractional Laplacian, contrary to what happens in \cite{Guo-Hamel}, we show that there exists a unique non-trivial positive bounded stationary state, supported everywhere in the domain. \textcolor{black}{Moreover, this steady state invades the unstable state $0$ with an exponential speed.}

\subsection{Assumptions, notations and results}

\textcolor{black}{The domain $\Omega$ is a smooth non-connected periodic domain of $\mathbb{R}^d$
\begin{equation}\label{Omega}
\text{i.e. }\ \Omega = \underset{k \in \mathbb{Z}^d}{\bigcup} \Omega_0 + a_k, \text{ with }\Omega_0 \text{ a smooth bounded domain of } \mathbb{R}^d \text{ and } a_k \in \mathbb{R}^d.
\end{equation}
We assume that 
\[\left(\Omega_0 +  a_{i}\right) \cap (\Omega_0 + a_j)  \neq \emptyset \quad  \text{ if and only if } \quad  i = j.\]
Moreover, if we denote $e_i$ the $i^{th}$ vector of the canonical basis of $\mathbb{R}^d$ \textcolor{black}{then we assume that for all $k \in \mathbb{Z}^d$ there holds} $a_{k + e_i} - a_k =  a_{e_i}$. \textcolor{black}{Moreover, }we assume that the principal eigenvalue $\lambda_1$ of the Dirichlet operator $(-\Delta)^\alpha - Id$ in $\Omega_0$ is negative
\begin{equation} \label{H1}
\tag{H1}
\text{i.e. } \quad \lambda_1< 0.
\end{equation} }
\textcolor{black}{We also introduce the eigenvalue problem associated to the whole domain $\Omega$.} It is well known (thanks to the Krein Rutman theorem) that the principal eigenvalue $\lambda_0$ of the Dirichlet operator $(-\Delta)^\alpha - Id$ in $\Omega$ is simple in the algebraic and geometric sense and moreover, the associated principal eigenfunction $\phi_0$, solves 
\begin{equation}\label{lambda0}
\mathrm{i.e.} \ 
  \left\{
      \begin{aligned}
       &\left( (-\Delta)^\alpha -Id\right) \phi_0=\lambda_0 \phi_0  && \ \text{ in } \Omega, \\
       &\phi_0 = 0  && \ \text{ in }  \Omega^c, \\
       & \phi_0 \text{ has a constant sign that can be chosen positive.}
      \end{aligned}
    \right. 
\end{equation}

The first result of this paper ensures the existence and the uniqueness of a positive bounded stationary state $n_+$ of \eqref{equation_1}:
\begin{equation}\label{equation_1_stat_1}
\mathrm{i.e.} \ 
	\left\lbrace
		\begin{aligned}
			(-\Delta)^\alpha n_+  &= n_+-n_+^2 && \text{ in }  \Omega,\\
			n_+& =0 && \text{ in } \Omega^c.
		\end{aligned}
	\right.
\end{equation}
\begin{theorem}\label{theorem_existence_uniqueness_stationnary_state}
Under the assumption \eqref{H1}, there exists a unique positive and bounded stationary state $n_+$ to \eqref{equation_1}. Moreover, we have $0 \leq n_+ \leq 1$ and $n_+$ is periodic.
\end{theorem}
The existence is due to the negativity of the principal eigenvalue of the Dirichlet operator $(-\Delta)^\alpha - Id$ in $\Omega_0$ which allows to construct by an iterative method a stationary state (see \cite{Smoller} for more details). As for the uniqueness, the main step is to prove that thanks to the non-local character of the fractional Laplacian, any positive bounded stationary state behaves like 
\begin{equation}\label{def_delta}
\delta(x)^\alpha= \mathrm{dist}(x, \partial \Omega)^\alpha 1_{\Omega}(x).
\end{equation}
\textcolor{black}{Then, a} classical argument (see \cite{Berestycki81} and \cite{BHR}) relying on the maximum principle and the Hopf lemma provides the result. We should underline that the uniqueness is clearly due to the non-local character of the operator $(-\Delta)^\alpha$, and it does not hold in the case of a standard diffusion term ($\alpha=1$). A direct consequence of the existence of a stationary solution is
\begin{corollary}\label{lambda0negatif}
The principal eigenvalue \textcolor{black}{$\lambda_0$} of the Dirichlet operator $(-\Delta)^\alpha - Id$ in $\Omega$ is negative. 
\end{corollary}

Once we have established a unique candidate to be the limit of $n(x,t)$ as $t$ tends to $+\infty$, we prove the invasion phenomena. First, we prove that starting from
\begin{equation}\label{H2}
\tag{H2}
n_0 \in C_0^\infty (\Omega) \cap C_c (\mathbb{R}^d) \quad \text{ and } \quad n_0\not\equiv 0
\end{equation}
the solution has algebraic tails at time $t=1$. To prove it, we provide an estimate of the heat kernel at time $t=1$ for a general multi-dimensional domain which satisfies the uniform interior and exterior ball condition:
\textcolor{black}{\begin{definition}[The uniform interior and exterior ball condition]
A set $\mathcal{O} \subset \mathbb{R}^d$ with $d \geq 1$ satisfies the uniform interior and exterior ball condition if there exists $r_1>0$ such that 
\begin{align*}
&\forall x \in \partial \mathcal{O}, \ \exists y_x \in \mathcal{O} \text{ such that } x \in \partial B(y_x, r_1) \text{ and } B(y_x, r) \subset \mathcal{O}, \\
\text{ and } & \forall z \in \mathcal{O}^c, \ \exists y_z \in \mathcal{O}^c \text{ such that } z \in B(y_z, r_1)  \text{ and } B(y_z, r) \subset \mathcal{O}^c. 
\end{align*}
\end{definition}}
\begin{theorem}\label{heat}
Let $\mathcal{O}$ be a smooth domain of $\mathbb{R}^d$ with $d \geq 1$ satisfying the uniform interior and excterior ball condition.
If we define $p$ as the solution of the following equation 
\begin{equation}\label{app_eq_1}
\left\lbrace                       
\begin{aligned}
&\partial_t p (x,t)+ (-\Delta)^\alpha p(x,t) =0 && \text{ for all } (x,t) \in  \mathcal{O}\times ]0,+\infty[,\\
& p(x,t)=0 && \text{ for all } (x,t) \in  \mathcal{O}^c \times [0,+\infty[,\\
&p(x,t=0)=n_0(x) \in C^\infty_0(\mathcal{O}, \mathbb{R}^+) \cap C^0_c(\mathbb{R}^d),
\end{aligned}
\right.
\end{equation}
then there exists $c>0$ and $C>0$ such that for all $x \in \mathcal{O} $, 
\begin{equation}
\frac{c \times  \delta(x)^\alpha }{1+|x|^{d+2\alpha}} \leq p(x,t=1) \leq \frac{C\times \delta(x)^\alpha}{1+|x|^{d+2\alpha}}.
\end{equation}
\end{theorem}
Once Theorem \ref{heat} is established, we are able to state the main result of the paper. 
\textcolor{black}{\begin{theorem} \label{theorem_convergence_1}
Assume \eqref{H1} and \eqref{H2}. Then for all $\mu >0$ there exists a time $t_\mu>0$ such that:\\
(i) for all $ c < \frac{|\lambda_0|}{d+2\alpha}$ and all $(x,t) \in \left\lbrace |x| < e^{ct }\right\rbrace \times ]t_\mu, +\infty[$
\[ |n(x,t) - n_+(x) | \leq \mu.\]
(ii) for all $ C > \frac{|\lambda_0|}{d+2\alpha}$ and all $(x,t) \in \left\lbrace |x| > e^{Ct} \right\rbrace \times ]t_\mu, +\infty[$
\[ |n(x,t) | \leq \mu.\]
\end{theorem}}
We detail the general strategy to prove Theorem \ref{theorem_convergence_1} in the next section.

\subsection{Discussion on the main results}

%\begin{equation}\label{resultat_heat}
%\frac{c\times \delta(x)^\alpha}{1+|x|^{1+2\alpha}} \leq n(x,t=1) \leq \frac{C}{1+|x|^{1+2\alpha}}.
%\end{equation}

\textcolor{black}{Theorem \ref{heat} is an application of general results about the fractional Dirichlet heat kernel estimates given for instance in \cite{Dirichlet-Heat-Kernel} or in \cite{Bogdan}. Both of the two cited articles use a probabilistic approach. We propose in this work a deterministic proof of the lower bound of the fractional Dirichlet kernel estimates. Our proof is quite simple but the result is not as general as those presented in \cite{Dirichlet-Heat-Kernel} and \cite{Bogdan}. \textcolor{black}{In particular, it is only valid for finite time.} It relies on a well adapted decomposition of the fractional Laplacian. We do not provide the proof of the upper bound of the fractional Dirichlet kernel estimates since there is no difficulties to obtain such bound. }

\bigbreak

\textcolor{black}{\noindent Theorem \ref{theorem_convergence_1} can be seen as a \textcolor{black}{generalisation} of the results of \cite{Coulon} or \cite{Leculier_1}. Indeed, if we study a non-local Fisher KPP equation in the whole domain $\mathbb{R}^d$ with a reaction term depending on a parameter such that the reaction term becomes more and more unfavorable in $\Omega^c$ then we recover Theorem \ref{theorem_convergence_1}. This is fully in the spirit of \cite{Guo-Hamel}. In fact, if we study the equation:
\begin{equation*}
%\tag{1_{\delta}}
\left\lbrace 
\begin{aligned}
&\partial_t n + (-\Delta)^\alpha n = \mu_\delta(x) n - n^2 && \text{ in } \mathbb{R}^d \times ]0, +\infty[, \\
& n(x, t= 0) = n_0(x),
\end{aligned}
\right.
\end{equation*}
with 
\begin{equation*}
\mu_\delta (x) = 
\left\lbrace 
\begin{aligned}
&1 && \text{ if } x \in \Omega, \\
&1 - (\delta + 1)\mathrm{dist}(x, \Omega) &&\text{ if } 0 < \mathrm{dist}(x, \Omega) \leq \frac{1}{\delta}, \\
&-\frac{1}{\delta} && \text{ if } \frac{1}{\delta} <  \mathrm{dist}(x, \Omega) .
\end{aligned}
\right.
\end{equation*}
Then, denoting by $\lambda_\delta$ the principal eigenvalue of the operator $((-\Delta)^\alpha - \mu_\delta)$ we claim that 
\begin{equation}
 \lambda_\delta \underset{ \delta \rightarrow 0}{\longrightarrow } \lambda_0.  
\end{equation}
\textcolor{black}{It is then possible to obtain the result of Theorem \ref{theorem_convergence_1} from such approximate problems in the spirit of \cite{Guo-Hamel}. Although we do not use such method, similar difficulties would arise to treat the problems with this approximation procedure. Our method can indeed be adapted to study those problems in a uniform way. }}

\section{Strategy, comparison tools and outline of the paper}

\subsection{The general strategy}

\textcolor{black}{The general strategy to establish the results of Theorem \ref{theorem_convergence_1} is the following:
\bigbreak
\indent A- \textcolor{black}{Identify} the unique candidate to be the limit. This is the content of Theorem \ref{theorem_existence_uniqueness_stationnary_state}. \\
\indent B- Starting from a compactly supported initial data, the solution $n$ has algebraic tails immediatly after $t=0$. This is the content of Theorem \ref{heat}. \\
\indent C- \textcolor{black}{Establish} a sub and a super-solution which bound the solution $n$ from below and above. \\
\indent D- Use the sub-solution to "push" the solution $n$ to the unique non-trivial stationary state $n_+$ in $\left\lbrace |x| < e^{\frac{|\lambda_0|t}{d+2\alpha}} \right\rbrace$ and use the super-solution to "crush" the solution $n$ to $0$ in $\left\lbrace |x| > e^{\frac{|\lambda_0|t}{d+2\alpha}} \right\rbrace$.}

\bigbreak

\textcolor{black}{The proof of C can be done with two different approaches. The first one is introduced in \cite{Cabre-Coulon-Roquejoffre} by Cabr\'{e}, Coulon and Roquejoffre. The idea is to consider the quantity
\[v(x,t) =\phi_0(r(t)x)^{-1} n(r(t)x, t)\]
where the eigenfunction $\phi_0$ is introduced in \eqref{lambda0} and $r(t)$ decreases exponentially fast. Next, \textcolor{black}{the problem can be formally reduced to a transport equation leading to the fact that $v$ is of the form} $\frac{\phi_0(x)}{1+b(t)|x|^{d+2\alpha}}$. The idea is therefore to look for a sub-solution $\underline{v}$ and a supersolution $\overline{v}$ of the form 
\[\underline{v}(x,t) = \frac{\underline{a}\phi_0(x)}{1+\underline{b}(t)|x|^{d+2\alpha}} \quad \text{ and } \quad \overline{v}(x,t) = \frac{\overline{a}\phi_0(x)}{1+\overline{b}(t)|x|^{d+2\alpha}}\]   
(where the positive constants $\underline{a}, \overline{a}$ and the function $\underline{b}, \overline{b}$ have to be adjusted). }

%We transform \eqref{equation_1} into a simpler transport equation involving also a fractional Laplacian. 

\textcolor{black}{The second approach is introduced in \cite{Mirrahimi1} by M\'{e}l\'{e}ard and Mirrahimi (in order to extend the singular perturbation approach of \textcolor{black}{\cite{Opt_geom_4} and \cite{Opt_geom_2}}, put to work in the \textcolor{black}{PDE} frame work in \cite{Opt_geom_1}). The main idea is to perform the following scaling on equation \eqref{equation_1}
\begin{equation}\label{scale}
(x,t) \mapsto \left(|x|^{\frac{1}{\varepsilon}} \frac{x}{|x|}, \frac{t}{\varepsilon} \right).
\end{equation}
The interest of this scaling is to catch the effective behavior of the solution. Indeed, this scaling lets invariant the set 
\[\mathcal{B}=\left\lbrace (x,t) \in \mathbb{R} \times \mathbb{R}^+ \  | \  (d+2\alpha) \log|x|< |\lambda_0|t  \right\rbrace\]
where $\lambda_0$ is defined by \eqref{lambda0}. Then, we look for sub/super-solutions on the form 
\[ \phi_0 (x) \times G(x,t)\]
where $G$ needs to be determined. Taking $G$ with an algebraic tail gives that, once the scaling is performed, the fractional Laplacian of $G$ vanishes as the parameter $\varepsilon$ tends to $0$. Therefore, the sub and super solutions are just perturbations of a simple ODE. }

\textcolor{black}{We choose the second method because it explains the main role of the fractional Laplacian, namely to provide algebraic tails. Once this tails are obtained in part B, the role of the fractional Laplacian becomes negligible. \textcolor{black}{This means that the only role of the fractional Laplacian in determining the invasion speed is at initial time where it determines the algebraic tails of the solution. This is indeed very different from the classical Fisher-KPP equation where the diffusion not only determines the exponential tails of the solution but it also modifies the invasion speed in positive times (see \cite{Mirrahimi1}). This is why in the asymptotic study of the classical Fisher KPP equation, one obtains a Hamilton-Jacobi equation \cite{Opt_geom_1} while in the fractional KPP equation the limit is a simple ordinary differential equation.}}

% The situation is similar in the standard Fisher KPP equation (i.e. with a standard Laplacian), if the initial data has algebraic tails then the solution invades exponentially fast. However, it is well known that in a standard Fisher KPP equation, if we start from a compactly supported data then the speed of invasion is constant (see \cite{HamelRoques} and \cite{Aronson-Weinberger}).}% It shows somehow that the main role of the fractional Laplacian is to provide the algebraic tails. 

%\textbf{Fast propagation for KPP equations with slowly decaying initial conditions Francois Hamel and Lionel Roques }
%More precisely, the scaling \eqref{scale} introduced by M\'{e}l\'{e}ard and Mirrahimi  

\bigbreak 

\textcolor{black}{The proof of D can be achieved with the rescaled solution $n\left(|x|^{\frac{1}{\varepsilon}}\dfrac{x}{|x|}, \dfrac{t}{\varepsilon}\right)$ using the method of perturbed test functions from the theory of viscosity solutions and homogenization (introduced by Evans in \cite{Evans_visco_1} and \cite{Evans_visco_2} and by Mirrahimi and M\'{e}l\'{e}ard in \cite{Mirrahimi1} for the fractional Laplacian). Since the proof is technical, long and not easy to grasp (the domain moves also with the parameter $\varepsilon$), we prefer to drop the scaling and to perform the inverse scaling on our sub and super solutions. Therefore, we provide a direct proof of D by adapting the proof of Theorem 1.6 in \cite{Coulon}. In this proof, the author proves thanks to a subsolution  that there exists $\sigma>0$ and $t_\sigma>0$ such that
\[\sigma < \underset{ (x,t) \in \left\lbrace |x| < e^{\frac{(|\lambda_0|-\delta)t}{d+2\alpha}}  \right\rbrace \times ]t_\sigma, +\infty[ }{\inf} n(x,t).\]
This last claim is obviously false in our case \textcolor{black}{since the solution vanishes on the boundary}. This is the main new difficulty that we will encounter. We overcome it by establishing the same kind of estimates away from the boundary.}% However, we underline that the idea of using the method of perturbed test functions from the theory of viscosity solutions and homogenization introduced in \cite{Mirrahimi1} (and used after in \cite{Bouin_al}, \cite{Souganidis_Tarfulea} and \cite{Leculier_1} in similar contexts) works. This proof will be provided in the thesis manuscript of the first author. 
\subsection{The comparison tools and some notations}

\textcolor{black}{All along the article, we will use many times the comparison principle. We recall here what we mean by comparison principle. 
\begin{theorem*}[The comparison principle]
Let $f$ be a smooth function, $a \in [0, +\infty[$ and $b \in ]0, +\infty]$. If $\underline{n}$ and $\overline{n}$ are such that 
\begin{align*}
&\forall (x,t) \in \Omega \times ]a, b[, \quad \partial_t \underline{n} + (-\Delta)^\alpha \underline{n} \leq f(\underline{n}), \quad &&\text{ and } \quad \partial_t \overline{n} + (-\Delta)^\alpha \overline{n} \geq f (\overline{n}), \\
& \forall (x,t) \in \Omega^c \times ]a,b[, \quad \underline{n} (x,t) \leq \overline{n}(x,t), \quad &&\text{ and }  \quad \forall x \in \mathbb{R}^d, \quad \underline{n}(x,t=a) \leq \overline{n}(x,t=a)
\end{align*}
then 
\[\forall (x,t) \in \Omega \times ]a, b[, \quad \underline{n}(x,t) \leq \overline{n}(x,t).\]
\end{theorem*}
In the same spirit, we recall the fractional Hopf Lemma stated in \cite{Hopf_lemma_Greco}. 
\begin{lemma*}[The fractional Hopf Lemma \cite{Hopf_lemma_Greco}]
Let $\mathcal{O} \subset \mathbb{R}^d$ be an open set satisfying the uniform interior and exterior ball condition at $z \in \partial \mathcal{O}$ and let $c \in L^\infty(\mathcal{O})$. Consider a positive lower semi-continuous function $u: \mathbb{R}^d \mapsto \mathbb{R}$ satisfying $(-\Delta)^\alpha u \geq c(x) u$ point-wise in $\Omega$. Then, either $u$ vanishes identically in $\Omega$, or there holds 
\[ \underset{x \in \Omega}{\underset{x  \mapsto z}{\liminf} }\  \frac{u(x)}{\delta(x)^\alpha} >  0.\] 
\end{lemma*}}

\bigbreak

All along the article, for any set $\mathcal{U}$ and any positive constant $\nu$, we introduce the following new sets :
\begin{equation}\label{setU}
\mathcal{U}_\nu = \left\lbrace x  \in  \mathcal{U} | \  \mathrm{dist}(x,\partial \mathcal{U})>\nu  \right\rbrace, \quad \mathcal{U}_{-\nu} = \left\lbrace x  \in  \mathbb{R}^d  | \  \mathrm{dist}(x, \mathcal{U})<\nu \right\rbrace.
\end{equation}
The constants denoted by $c$ or $C$ may change from one line to another when there is no confusion possible. Also, we drop the constant $C_\alpha$ and the Cauchy principal value $P.V.$ in front of the fractional Laplacian for better readability.

\subsection{Outline of the paper}

In section 3, we demonstrate Theorem \ref{theorem_existence_uniqueness_stationnary_state}. Next, section 4 is dedicated to the proof of Theorem \ref{heat}. The first part of section 5 introduces the scaling and provides the sub and super-solutions. Finally, the second part of section 5 is devoted to the proof of Theorem \ref{theorem_convergence_1}.

%\section{Strategy and organization of the paper}

\section{Uniqueness of the stationary state $n_+$}

First, we state a proposition which gives the shape of any non-trivial bounded sub and super-solution to \eqref{equation_1_stat_1} near the boundary. Then, we use this result to prove the uniqueness result. Since the proof of the existence is classical we do not provide it. %It relies on an increasing sequences initialized by $\varepsilon \phi_A$ (see \cite{Smoller} for more details). 

\begin{proposition}\label{proposition:shape}
\textit{(i)} If $u$ is a smooth positive bounded function such that $u(x)=0 $ for all $x \in \Omega^c$ and $(-\Delta)^\alpha u(x) \leq u(x)-u(x)^2$ for all $x \in \Omega$, then there exists $C >0$ such that for all $x \in \mathbb{R}^d$
\[ u(x)\leq C \delta(x)^\alpha.\]

\textit{(ii)} If $v$ is a smooth positive bounded function such that $v(x)=0 $ for all $x \in \Omega^c$, $(-\Delta)^\alpha v(x) \geq v(x)-v(x)^2$ for all $x \in \Omega$ and $v\not \equiv 0$ then there exists $c >0$ such that for all $x \in \mathbb{R}^d$ 
\[ c \delta(x)^\alpha \leq v(x) . \]
\end{proposition}
%Before giving the proof of Proposition \ref{proposition:shape}, we recall a lemma proved in  which states a useful barrier function.
%\begin{lemma}[\cite{Ros-Oton-Serra}]\label{lemma_super}
%There exists $C>0$ and a radial continuous function $\overline{\psi} \in H^1_{loc} (\mathbb{R})$ satisfying :
%\begin{equation}\label{equation_super_solution}
%  \left\{
%      \begin{aligned}
%       & (-\Delta)^{\alpha} \overline{\psi}\geq 1 && \text{ in } ]-4,4[\backslash ]-1,1[,\\
%       &\overline{\psi} \equiv 0 && \ \text{ in } ]-1,1[,\\
%       &0 \leq \overline{\psi} \leq C|x-1|^\alpha && \text{ in } ]-4,4[ \backslash ]-1,1[, \\
%       & 1 \leq \overline{\psi} \leq C &&\text{ in } \mathbb{R} \backslash ]-4,4[.
%      \end{aligned}
%    \right. 
%\end{equation}
%
%\end{lemma}
%

\begin{proof}[Proof of Proposition \ref{proposition:shape}]

\textit{Proof of (i).} Let $u$ be a continuous positive bounded function such that $u=0$ in $\Omega^c$ and $(-\Delta)^\alpha u \leq u-u^2$ in $\Omega$. Let $x$ be a point of the boundary. Let $z_x\in \mathbb{R}^d$ and $r_1>0$ be the elements provided by the uniform exterior ball condition such that 
\[ B(z_x, r_1) \subset \Omega^c  \quad \text{ and } \quad x \in \overline{B(z_x,r_1) \cap \partial \Omega } .\]
 We rescale and translate a barrier function (provided for instance in Annex B of \cite{Ros-Oton-Serra}).
%\[\overline{\phi}(x):=\max(1,\left(\frac{A}{8}\right)^{2\alpha}, \max u)  \ \overline{\psi}(\frac{8}{A}x+1).\]
This barrier function $\overline{\phi}$ satisfies the following properties: 
\begin{equation}
  \left\{
      \begin{aligned}
       & (-\Delta)^{\alpha} \overline{\phi}\geq 1 && \text{ in } B(z_x, 4r_1) \backslash B(z_x, r_1),\\
       &\overline{\phi} \equiv 0 && \text{ in } B(z_x,r_1),\\
       &0 \leq \overline{\phi} \leq C(|z_x - x|-r)^\alpha && \text{ in } B(z_x, 4r_1) \backslash B(z_x, r_1), \\
       & \max \ u \leq \overline{\phi} \leq C  &&\text{ in } \mathbb{R}^d \backslash B(z_x,4r_1).
      \end{aligned}
    \right. 
\end{equation}
%\max(1,\left(\frac{A}{8}\right)^{2\alpha}, \max u)
\textcolor{black}{We prove that $u \leq \overline{\phi}$ in $\mathbb{R}^d$. }By construction we have $u \leq \overline{\phi} $ in $(B(z_x, 4r_1) \backslash B(z_x, r_1))^c$. 
%We want to prove by contradiction that $u \leq \overline{\phi}$ in $[0, \frac{3A}{8}]$.
Assume by contradiction that there exists $x_0 \in (B(z_x, 4r_1) \backslash B(z_x, r_1)) \cap \Omega$ such that $(\overline{\phi} -u)(x_0) <0$. Then, there exists $x_1 \in (B(z_x, 4r_1) \backslash B(z_x, r_1)) \cap \Omega$ such that $(\overline{\phi} -u)(x_1) = \underset{x \in \mathbb{R}^d}{\min} (\overline{\phi} -u)(x)<0$. Thus, we obtain 
\[(-\Delta)^\alpha (\overline{\phi} -u)(x_1)<0 \quad \text{ and } \quad (-\Delta)^\alpha (\overline{\phi} -u)(x_1)\geq  1-u(x_1) + u(x_1)^2 \geq 0,\]
a contradiction.

\bigbreak

\textit{Proof of (ii)}. Let $v$ be a continuous positive bounded function such that $v=0$ in $\Omega^c$ and $(-\Delta)^\alpha v \geq v-v^2$ in $\Omega$.
An easy but important remark is the following: thanks to the non-local character of the fractional Laplacian, since $v \not \equiv 0$, we deduce that $v>0$ in the whole domain $\Omega$. Otherwise, the following contradiction holds true :
\[\exists \underline{x} \in \Omega \text{ such that } v(\underline{x}) = 0 \quad  \text{ and } \quad (-\Delta)^\alpha v (\underline{x}) - v(\underline{x}) + v(\underline{x})^2 = - \int_{\mathbb{R}^d} \frac{v(y)}{|x-y|^{d+2 \alpha}}dy <0.\]

Next, let $k$ be any element of $\mathbb{Z}^d$. We introduce  $\underline{w}_k : (x,t) \in \mathbb{R}^d \times [0, +\infty[ \mapsto \underline{w}_k (x,t) \in \mathbb{R}$ as the solution of 
\begin{equation}\label{equation:w_k}
  \left\{
      \begin{aligned}
       &\partial_t \underline{w}_k + (-\Delta)^{\alpha} \underline{w}_k=\underline{w}_k-\underline{w}_k^2 &&\text{ in } (\Omega_0  + a_k)\times ]0,+\infty[,\\
      &\underline{w}_k(x,t) =0&&  \text{ in } \mathbb{R}^d \backslash (\Omega_0  + a_k) \times [0,+\infty[ \\
			&\underline{w}_k(x,0)=v(x) && \text{ in } (\Omega_0  + a_k),
      \end{aligned}
    \right. 
\end{equation}
\textcolor{black}{where $\Omega_0$ and $a_k$ are introduced in \eqref{Omega}.} Thanks to the remark above, and recalling \eqref{H1}, we deduce thanks to Theorem 5.1 in \cite{Roquejoffre1} that $\underline{w}_k(.,t) \underset{ t \rightarrow +\infty }{\longrightarrow}  \underline{w}_{stat}(.)$ with $\underline{w}_{stat}$ the solution of
\begin{equation} \label{equation_stationnaire}
  \left\{
      \begin{aligned}
       & (-\Delta)^{\alpha} \underline{w}_{stat}=\underline{w}_{stat}-\underline{w}_{stat} ^2,&& \text{ in } (\Omega_0  + a_k),\\
       &\underline{w}_{stat} =0&&  \text{ in } \mathbb{R}^d \backslash (\Omega_0  + a_k).\\
      \end{aligned}
    \right. 
\end{equation}
Note that the above $\underline{w}_{stat}$ does not depend on the choice of $k$, i.e. $\underline{w}_k(\cdot , t)$ converges as $t$ tends to $+\infty$ to the same $\underline{w}_{stat}$ (up to a translation). Then, we conclude thanks to the comparison principle that 
\[\underline{w}_{stat}(x) \leq v(x), \ \forall x\in \mathbb{R}^d . \]
Since, $(\Omega_0  + a_k)$ is bounded, we apply the results of \cite{Ros-Oton-Serra} to find that there exists a constant $c>0$ such that
\begin{equation*}
c\delta(x)^\alpha1_{(\Omega_0  + a_k)}(x) \leq \underline{w}_{stat}(x) \leq v(x).
\end{equation*}
The previous analysis holds for every $k\in \mathbb{Z}^d$. We conclude that 
\begin{equation}\label{g}
c \delta (x)^\alpha \leq v(x).
\end{equation}
\end{proof}

\begin{proof}[Proof of Theorem \ref{theorem_existence_uniqueness_stationnary_state}] 
The argument relies on the fact that two steady solutions are comparable everywhere thanks to Proposition \ref{proposition:shape}. This is in the spirit of \cite{Berestycki81} and \cite{BHR} in the  context of standard diffusion. Let $u$ and $v$ be two bounded steady solutions of \eqref{equation_1_stat_1}. By the maximum principle, we easily have that for all  $x \in \mathbb{R}^d$, 
\begin{equation*}
u(x) \leq 1 \ \text{ and } v(x)\leq 1.
\end{equation*}
We will assume that 
\begin{equation}\label{evaluer_1/2_2}
v(x_0)\leq u(x_0) \qquad \text{ where } x_0 \in \Omega_0.
\end{equation}
Thanks to Proposition \ref{proposition:shape}, we deduce the existence of two constants $0 < c \leq C$ such that:
\[c \delta(x)^\alpha \leq u(x) \leq C \delta (x)^\alpha \ \text{ and } \ c \delta(x)^\alpha \leq v(x) \leq C \delta (x)^\alpha. \]
Thus there exists a constant $\lambda >1$ such that for all $x \in \mathbb{R}^d$, 
\begin{equation}
u(x) \leq \lambda v(x).
\end{equation}
We set $l_0:=\inf \left\lbrace \lambda \geq 1|\  \forall x\in\mathbb{R}^d, \  u(x) \leq \lambda v(x)\right\rbrace$. The point is to prove by contradiction that $l_0=1$. It implies that $x_0$ is a contact point, and will allow us to conclude thanks to the fractional maximum principle that $u=v$.\\
We assume by contradiction that $l_0>1$. Next, we define : 
\begin{equation}
\widetilde{w} = \underset{x  \in \Omega}{\inf}\ \dfrac{(l_0 v -u)(x)}{\delta(x)^\alpha} \geq 0. 
\end{equation}
There are two cases to be considered.

\textbf{Case 1:} $\widetilde{w} > 0$.\\
We show in this case that we can construct $l_1 < 1$ such that $u(x) \leq l_1 l_0 v(x)$ for all $x \in \mathbb{R}^d$ : a contradiction. If $\widetilde{w} > 0$, we claim that there exists $\mu\in ]0,1[$ and $\nu> 0$ such that for all $x \in \Omega \backslash \Omega_\nu$ (we recall that $\Omega_\nu$ is defined by \eqref{setU}), 
\begin{equation}
\dfrac{\widetilde{w}}{2} \leq \dfrac{(\mu l_0 v -u) (x)}{\delta(x)^\alpha} .
\end{equation}
Indeed, if there does not exist such couple $(\mu, \nu)$, we deduce that for all $n \in \mathbb{N}$, there exists $(x_n)_{n \in \mathbb{N}} \in \Omega$, such that $\delta(x_n) \leq \frac{1}{n}$ and 
\begin{equation*}
\dfrac{((1-\frac{1}{n})l_0 v -u) (x_n)}{\delta(x_n)^\alpha}< \dfrac{\widetilde{w}}{2}.
\end{equation*}
Passing to the liminf we get the following contradiction :
\[0<\widetilde{w}\leq \dfrac{\widetilde{w}}{2}.\]
And so, the existence of the couple $(\mu, \nu)$ implies that
\begin{equation}\label{contradiction_w_positif_1}
(\mu l_0 v - u ) (x) \geq 0,  \ \forall x \in \Omega \backslash \Omega_\nu.
\end{equation}

Next, we claim that %there exists $\rho \in ]0,1[ $ such that for all $x \in \Omega_\nu$, 
\begin{equation}
\exists \rho >0 \text{ such that } \forall x \in \Omega_\nu, \text{ we have } \rho \leq (l_0 v - u )(x). 
\end{equation}
Indeed, if such $\rho$ does not exist then there exists a sequence $(x_n)_{n \in \mathbb{N}} \in \Omega$ such that $\delta(x_n) \geq \nu$ and $(l_0v - u )(x_n) \underset{ n \rightarrow +\infty}{\longrightarrow} 0$. Then we obtain 
\[  \frac{(l_0 v - u )(x_n)}{\delta(x_n)^\alpha} \leq \frac{(l_0 v - u )(x_n)}{\nu^\alpha} \underset{n \rightarrow +\infty}{\longrightarrow} 0\]
which is in contradiction with the hypothesis $\widetilde{w}>0$. The existence of such $\rho$ implies that for all $x \in \Omega_\nu$
\begin{equation}\label{contradiction_w_positif_2}
\left((1-\frac{\rho}{\max l_0v }) l_0 v - u \right)(x) \geq 0.
\end{equation}
Finally, if we define $l_1 = \max (\mu ,  1-\frac{\rho}{\max l_0v +1 })$ then we obtain the desired contradiction. Therefore this case cannot occur. 
\bigbreak

\textbf{Case 2:} $\widetilde{w}=0$.\\
We consider $(x_n)_{n \in \mathbb{N}}$ a minimizing sequence of $\widetilde{w}$. There are 3 subcases : a subsequence of $(x_n)_{n \in \mathbb{N}}$ converges to $x_0 \in \Omega$, a subsequence of $(x_n)_{n \in \mathbb{N}}$ converges to $x_b \in \partial \Omega$ and any subsequence of $(x_n)_{n \in \mathbb{N}}$ diverges. 

\textbf{Subcase a: } \textit{ There exists }$ x_0 \in \Omega, \  such \ that \  \frac{(l_0 v-u)(x_0)}{\delta(x_0)^\alpha}=0$.\\
Since $x_0 \in \Omega$ we deduce that $(l_0v-u)(x_0)=0$. Hence, by the maximum principle, $u=l_0 v$. We deduce that $l_0 v$ is a solution of \eqref{equation_1_stat_1} and we conclude that :
\begin{equation}
l_0 ( v -v^2)= l_0 (-\Delta)^{\alpha}(v)=(-\Delta)^\alpha (l_0 v) = l_0 v- (l_0 v)^2.
\end{equation}
This equation leads to $l_0=1$, a contradiction.
\bigbreak
\textbf{Subcase b: }\textit{ There exists} $ x_b \in \partial\Omega, \  such \ that \  \underset{x \in \Omega}{\underset{x \rightarrow x_b,}{\liminf}} \ \frac{(l_0 v-u)(x)}{\delta(x)^\alpha}=0$.\\
Here is a summary of what we know:
\begin{align*}
&   (i) \  l_0 v -u \geq 0, \\
&   (ii) \  (-\Delta)^\alpha  ( l_0 v -u)\geq -l_0(l_0 v -u) ,\\
&   (iii)\   (l_0 v -u)(x_b)=0.
\end{align*}
According to the fractional Hopf Lemma, the previous assumptions leads to $\underset{ x \in \Omega}{\underset{x \rightarrow x_b,}{\liminf}} \  \frac{(l_0 v - u)(x)}{\delta(x)^\alpha} >0$. However, we have assumed that $\underset{ x \in \Omega}{\underset{x \rightarrow x_b,}{\liminf}} \  \frac{(l_0 v - u)(x)}{\delta(x)^\alpha} =0$, a contradiction. 
\bigbreak
\textbf{Subcase c: } \textit{There exists a minimizing sequence $(x_n)_{n \in \mathbb{N}}$ such that $|x_n|$ tends to the infinity.}\\
First, we set 
\[ \overline{x}_k= x_k- a_{\lfloor x_k  \rfloor},\]
where $\lfloor x \rfloor \in \mathbb{Z}^d$ is such that $x \in \Omega_0 + a_{\lfloor x \rfloor}$. Since $\overline{x}_k \in \Omega_0$, we deduce that up to a subsequence $\overline{x}_k$ converges to $\overline{x}_\infty \in \overline{\Omega_0}$. Then we define:
\[u_k(x)=u(x+\overline{x}_k) \text{ and } v_k(x)=v(x+\overline{x}_k).\]
We also define the following set : 
\[\Omega_\infty=\left\lbrace x \in \mathbb{R} \ | \ x +\overline{x}_\infty \in \Omega \right\rbrace.\]
By fractional elliptic regularity (see \cite{Ros-Oton-Serra_2}), we deduce that up to a subsequence $(u_n)_{n \in \mathbb{N}}$ and $(v_n)_{n\in \mathbb{N}}$ converges to $u_\infty$ and $v_\infty$ solutions that verifies 
%
%For every compact set $K$ of $\Omega_\infty$, there exists $n_0 \in \mathbb{N}$ such that 
%\[\forall n \geq n_0, \ \forall x \in K, \ x+x_n \in \Omega_\infty.\]
%Thus, for all $n \geq n_0$ and all $x \in K$ 
%\[(-\Delta)^\alpha u_n (x) = u_n (x)-u_n(x)^2 \quad \text{ and } \quad  (-\Delta)^\alpha v_n(x) = v_n(x)-v_n(x)^2 .\]
%According to \cite{Ros-Oton-Serra_2}, we deduce that the sequences $(u_n)_{n \in \mathbb{N}}$ and $(v_n)_{n \in \mathbb{N}}$ converge up to a subsequence locally uniformly to $u_\infty$ and $v_\infty$.% in $C^\beta(\Omega_\infty)$ with some $\beta>2\alpha$.
% We conclude that for all $x \in K$
%\[(-\Delta)^\alpha u_\infty(x)=u_\infty(x)-u_\infty(x)^2 \quad \text{ and } \quad (-\Delta)^\alpha v_\infty(x) = v_\infty(x)-v_\infty(x)^2.\]
%Since it is true in every compact subset of $\Omega_\infty$, it follows that for all $x \in \Omega_\infty$
\begin{align*}
\forall x \in \Omega_\infty, &\quad (-\Delta)^\alpha u_\infty(x)=u_\infty(x)-u_\infty(x)^2,  \qquad  (-\Delta)^\alpha v_\infty(x) = v_\infty(x)-v_\infty(x)^2\\
\text{ and } \  \forall x \in \Omega_\infty^c, &\quad u_\infty(x) = v_\infty(x) = 0.
\end{align*}
Remark that 
\[ l_0 v_\infty - u_\infty \geq 0 \ \text{ and } \ \underset{x \in \Omega_\infty}{\underset{x \rightarrow 0}{\liminf}}\ \frac{(l_0 v_\infty - u_\infty)(x)}{\mathrm{dist}(x, \partial \Omega_\infty)^\alpha}=0.\]
Hence, if $\overline{x}_\infty \in \Omega_0$ then $0 \in \Omega_\infty$ and we fall in the subcase a). If $\overline{x}_\infty \notin \Omega_0$ then $0\in \partial \Omega_\infty$ and we fall in the subcase b). Both cases lead to a contradiction. 

Thus, we conclude that $l_0=1$. 
\end{proof}

\begin{remark}
Noticing that for all $(x,k) \in \Omega \times \mathbb{Z}^d$, we have
\begin{equation*}
(-\Delta)^\alpha (n_+(.+a_k))(x) = \int_{\mathbb{R}} \frac{n_+(x+a_k) - n_+(y+a_k)}{|x +a_k -( y +a_k) |^{d+2\alpha}}dy = n_+(x+a_k)-n_+(x+a_k)^2,
\end{equation*}
we deduce by uniqueness of the solution of \eqref{equation_1_stat_1} that $n_+$ is periodic. 
\end{remark}

\section{The fractional heat kernel and the preparation of the initial data}
We first introduce some requirements in order to achieve the proof of the lower bound of Theorem \ref{heat}. Once we have established Theorem \ref{heat}, we apply it to the initial data. Let $u \in C^\infty(\mathbb{R}^d \times ]0, +\infty[)$, then we set for all $(x,t) \in \mathbb{R}^d \times ]0, +\infty[$ 
\begin{equation}\label{L}
L^\alpha(u)(x,t) = \int_{B(0, \nu)} \frac{u(x,t) - u(y,t)}{|y|^{d+2\alpha}}dy.
\end{equation}
We also introduce $\widetilde{\phi}_\nu$ as the principal positive eigenfunction of the operator $L^\alpha$ associated to the principal eigenvalue $\mu_\nu$
\begin{equation*}
\text{ i.e. }
\left\lbrace
\begin{aligned}
& L^\alpha \widetilde{\phi}_\nu =\mu_\nu \widetilde{\phi}_\nu && \text{ in }  B(0,\nu)\\
&\widetilde{\phi}_\nu = 0 && \text{ in } B(0,\nu)^c,\\
&\widetilde{\phi}_\nu \geq 0, \ \|\widetilde{\phi}_\nu \|_\infty = 1.  
\end{aligned}
\right.
\end{equation*}
Next, we state two intermediate technical results.

\begin{lemma}\label{dernierlemme}
Let $w$ be the solution of the equation
\begin{equation}\label{eq_4}
\left\lbrace
\begin{aligned}
&\partial_t w+ L^\alpha w =1 &&  \text{ in }  B(0,\nu)\times ]0,+\infty[\\
&w(x,t) = 0 && \text{ in }  B(0,\nu)^c\times [0,+\infty[ , \\
&w(x,t=0) =  0 && \text{ in } B(0,\nu).
\end{aligned}
\right.
\end{equation}
Then there exists a constant $c_\nu >0$ such that 
\[ c_\nu \times \widetilde{\phi}_\nu(x) \leq w(x,t=1) .\]
\end{lemma} 

\begin{proof}
We define $\tau(t)=\dfrac{1}{\mu_\nu} ( 1-e^{-\mu_\nu t}) $ such that 
\begin{equation*}
\left\lbrace
\begin{aligned}
&\tau'(t) + \mu_\nu \tau(t) =1,\\
&\tau(0)=0.
\end{aligned}
\right.
\end{equation*}
Thanks to this choice of $\tau(t)$, the application $\underline{w}(x,t):= \tau(t) \times \widetilde{\phi}_\nu(x)$ is a sub-solution to \eqref{eq_4}. Actually, we have 
\[(\partial_t + L^\alpha)(\underline{w})-1= \tau'\widetilde{\phi}_\nu + \mu_\nu \tau \widetilde{\phi}_\nu - 1 \leq \tau' \widetilde{\phi}_\nu + \mu_\nu \tau \widetilde{\phi}_\nu - \widetilde{\phi}_\nu=\widetilde{\phi}_\nu(\tau' + \mu_\nu \tau -1)=0.\]
Since $\underline{w}(t=0)=0 \leq w(t=0)$, we can conclude thanks to the comparison principle that for all $(x,t) \in \mathbb{R}^d \times [0, +\infty[$, we have $\underline{w}(x,t) \leq w(x,t)$. Setting the time $t=1$ in the last inequality leads to 
\[\underline{w}(x,1)=\dfrac{1}{\mu_\nu}(1-e^{-\mu_\nu})\phi_\nu(x) = c_\nu \phi_\nu(x) \leq w(x,1).\]
\end{proof}
Next, we establish a barrier function for $L^\alpha$ in the spirit of the one introduced in \cite{Ros-Oton-Serra}.
\begin{lemma}\label{barrier_function_lemma}
There exists a function $\underline{\psi}$ such that 
\begin{equation}\label{BarrierFunctionL}
\left\lbrace
\begin{aligned}
&L^\alpha \underline{\psi} \leq 0 && \text{ in } B(0,\nu) \backslash B(0,\frac{\nu}{2}), \\
&\underline{\psi} =0 && \text{ in } B(0,\nu) ^c, \\
& \underline{\psi} \leq 1 && \text{ in }  B(0,\frac{\nu}{2}), \\
& \underline{c} (\nu-|x|)^\alpha \leq \underline{\psi}   && \text{ in } B(0,\nu) , \\
&\underline{\psi} \text{ is continuous in } B(0,\nu) \backslash B(0,\frac{\nu}{2}).
\end{aligned}
\right.
\end{equation}
\end{lemma}
\begin{proof}
Choose $C$ large enough such that the first point and the third point of \eqref{BarrierFunctionL} holds true with the following $\underline{\psi}$:
\[\underline{\psi}(x) := \left( \dfrac{1}{C}  (\nu^2-|x|^2)^\alpha + \frac{1}{2} 1_{B(0,\frac{\nu}{4})}(x) \right) 1_{B(0,\nu)}(x).\]
Indeed, defining $f(x) := (\nu^2-|x|^2)^\alpha$,  we have for $C$ large enough and $x\in B(0,\nu) \backslash B(0, \frac{\nu}{2})$
\[L^\alpha \underline{\psi}(x) \leq  \frac{ L^\alpha f(x)}{C} -\frac{1}{2}\int_{B(0, \frac{\nu}{4})} \frac{1}{|x-y|^{d+2\alpha}}dy \leq  \frac{\underset{B(0,\nu) \backslash B(0,\frac{\nu}{2})}{\sup} | L^\alpha f |}{C} - \frac{m(B(0,\frac{\nu}{4}))}{2 } \times  \left( \frac{4}{\nu} \right)^{d+2\alpha} < 0.\]
The other conditions follow. 
\end{proof}

\begin{proof}[Proof of Theorem \ref{heat}]

%Thanks to \eqref{Condition:initial:heat:kernel}, we can focus our study on $x\in \mathcal{O} \backslash B(0,1)$. 
The aim is to prove that there exists a constant $c>0$ such that 
\begin{equation}\label{goalHeat}
\forall x \in \mathcal{O} , \text{ we have } \frac{ c \delta (x)^\alpha }{1+|x|^{d+2\alpha}} \leq p(x,1) .
\end{equation}
To achieve the proof, there will be 4 steps.\\
\textcolor{black}{First, up to a translation and possibily a scaling of $n$, we prove \eqref{goalHeat} in $\left\lbrace |x|< 1+2\nu \right\rbrace$ where $\nu = \min(\frac{1}{4}, r_1)$ (with $r_1$ the radius provided by the uniform interior ball).} Next, we introduce a suitable decomposition of the fractional Laplacian (involving $L^\alpha$) to prove the existence of $c_1>0$ such that
\textcolor{black}{\begin{equation}\label{eq_2}
\left\lbrace
\begin{aligned}
& \frac{c_1}{1+|x|^{d+2\alpha}} \leq \partial_t p(x,t) + L^\alpha p(x,t) +\lambda p(x,t) && \text{ for all } (x,t) \in  \textcolor{black}{\left( \mathcal{O} \backslash \left\lbrace |x| > 1+\nu \right\rbrace \right)} \times ]0,1], \\
&p(x,t)\geq 0 && \text{ for all } (x,t) \in \left( \mathcal{O} \backslash \left\lbrace |x| > 1+\nu \right\rbrace \right) ^c  \times [0,1],\\
&p(x,t=0)=n_0(x) \in C^\infty_0(\mathcal{O}, \mathbb{R}^+)\cap C_c(\mathbb{R}^d)
\end{aligned}
\right.
\end{equation}
where $L^\alpha$ is defined by \eqref{L} and $\lambda = \int_{\mathbb{R}^d \backslash B(0,\nu)} \frac{1}{|y|^{d+2\alpha}}dy$.
In a third step, we will show that }
\textcolor{black}{\begin{equation}\label{step2heat}
\exists c_2>0 \text{ such that } \frac{c_2}{1+|x|^{d+2\alpha}} \leq p(x,t=1) \text{ for all } x \in \left( \Omega_\nu \cap \left\lbrace |x| > 1+2\nu \right\rbrace \right) .
\end{equation}}
Finally, we prove the same kind of result near the boundary :
\textcolor{black}{\begin{equation}\label{step3heat}
\exists c_3>0 \text{ such that } \frac{c_3 \delta(x)^\alpha}{1+|x|^{d+2\alpha}} \leq p(x,t=1) \text{ for all } x \in \left(\mathcal{O} \backslash \mathcal{O}_\nu \cap \left\lbrace |x| > 1+2\nu \right\rbrace  \right) .
\end{equation}}

\textbf{Step 1.}
First, note that thanks to a translation and possibly a scaling, we can suppose the following hypothesis:  
\begin{equation}\label{Condition:initial:heat:kernel}
\exists \sigma > 0 \text{ such that } \sigma < n_0(x)\ \text{ for all } x \in B(0,2).
\end{equation} 
Next, we claim that 
\begin{equation}\label{infp}
\underset{z \in B(0,1+2\nu)}{\underset{t \in (0,1) }{\inf}}p(z,t)>0.
\end{equation}
Indeed, let $\phi_2$ be the first positive eigenfunction of the Dirichlet fractional Laplacian in $B(0,2)$ and $\lambda_2$ the associated eigenvalue
\begin{equation*}
\mathrm{i.e.} \ 
\left\lbrace
\begin{aligned}
 (-\Delta)^\alpha\phi_2 &=\lambda_2 \phi_2 && \ \text{ for } \ x \in B(0,2), \\
\phi_2 &= 0&& \ \text{ for } x \in B(0,2)^c, \\
\| \phi_2 \|_\infty& =1.
\end{aligned}
\right.
\end{equation*}
Then the function 
\[ \underline{p}(x,t) := \sigma \times \phi_2(x) \times e^{-\lambda_2 t}\] 
is a sub-solution to \eqref{app_eq_1} (where $\sigma$ is defined by \eqref{Condition:initial:heat:kernel}). According to the comparison principle, we have for all $(x,t) \in B(0,1+2\nu) \times [0,1]$ 
\[0 <\underset{y \in B(0,1+2\nu)}{\underset{s \in \  [0,1]}{\min}} \underline{p}(y,s) =  \sigma \times  \underset{B(0,1+2\nu)}{\min} \phi_2 \times e^{-|\lambda_2|} \leq \underline{p}(x,t)\leq p(x,t).\] 
\textcolor{black}{We deduce that if $c$ is small enough, then \eqref{goalHeat} holds true for all $x \in B(0, 1+2\nu)$. }

\textbf{Step 2. } In this step we prove \eqref{eq_2} which is a key element to prove \eqref{goalHeat} for $x \in \left( \left\lbrace |x| > 1+2\nu \right\rbrace \cap \Omega \right)$. %Therefore, we want to prove \eqref{goalHeat} for $x \in \left( \Omega \cap \left\lbrace |x| > 1+2\nu \right\rbrace \right)$. Nevertheless, for technical reasons, we need to establish \eqref{eq_2} for $x \in  \left( \Omega \cap \left\lbrace |x| > 1+\nu \right\rbrace \right)$.

Then, we focus on $\left\lbrace |x|>1+\nu \right\rbrace$. We split the fractional Laplacian into 2 parts:
\begin{equation}\label{SplittingDeltaAlpha}
(-\Delta)^\alpha p(x,t)=\int_{\mathbb{R}^d\backslash B(0,\nu)} \dfrac{p(x,t)-p(x+y,t)}{|y|^{d+2\alpha}}dy+ L^\alpha p(x,t) = I_1(x,t) + L^\alpha p(x,t).
\end{equation}
For $I_1$, we obtain : 
\[ I_1(x,t)=\int_{\mathbb{R}^d \backslash B(0,\nu)} \dfrac{p(x,t)-p(x+y,t)}{|y|^{d+2\alpha}}dy =\lambda p(x,t) - \int_{\mathbb{R}^d \backslash B(0,\nu)} \dfrac{p(x+y,t)}{|y|^{d+2\alpha}}dy.\]
Since $|x| > 1+\nu$, we have 
\begin{equation}\label{app_ineq_1}
\underset{z \in B(0,1+\nu)}{\underset{t \in (0,1) }{\inf}}p(z,t) \int_{B(0,1)}\dfrac{1}{|z-x|^{d+2\alpha}}dz \leq  \int_{B(-x,1)}\dfrac{p(x+y,t)}{|y|^{d+2\alpha}}dy \leq \int_{\mathbb{R}^d \backslash B(0,\nu)} \dfrac{p(x+y,t)}{|y|^{d+2\alpha}}dy .
\end{equation}
Equation \eqref{app_ineq_1} ensures the existence of a positive constant $c_1>0$ such that for all \textcolor{black}{$(x,t) \in \left( \Omega \cap \left\lbrace |x| > 1+\nu \right\rbrace \right) \times [0,1[$} we have
\begin{equation*}
 \frac{c_1}{1+|x|^{d+2\alpha}} \leq \int_{\mathbb{R}^d \backslash B(0,\nu)} \dfrac{p(x+y,t)}{|y|^{d+2\alpha}}dy.
\end{equation*}
It follows that 
\begin{equation}\label{I1final}
I_1(x,t) \leq \lambda p(x,t) - \frac{c_1}{1+|x|^{d+2\alpha}}.
\end{equation}
Equations \eqref{SplittingDeltaAlpha} and \eqref{I1final} lead to \eqref{eq_2}. Moreover, if we define $v(x,t)=e^{\lambda t}\times p(x,t)$, we find the following system:
\textcolor{black}{\begin{equation}\label{eq_3}
\left\lbrace
\begin{aligned}
&\frac{c_1}{1+|x|^{d+2\alpha}} \leq \partial_t v(x,t) + L^\alpha v(x,t)  && \text{ for } (x,t) \in    \left( \Omega \cap \left\lbrace |x| > 1+\nu \right\rbrace \right) \times ]0,1], \\
& v(x,t)\geq 0 && \text{ for } (x,t) \in \left( \Omega \cap \left\lbrace |x| > 1+\nu \right\rbrace \right)^c\times [0,1] , \\
&v(x,t=0)=n_0(x) \in C^\infty_0(\Omega, \mathbb{R}^+).
\end{aligned}
\right.
\end{equation}}

\bigbreak
\textbf{Step 3.} \textcolor{black}{By uniform continuity of $\left( x \mapsto \frac{1}{1+|x|^{d+2\alpha}} \right)$ in $\mathbb{R}^d$, we deduce the existence of $c_1'>0$ such that for all $x_0 \in \left( \mathcal{O}_\nu \cap \left\lbrace |x| > 1+\nu \right\rbrace \right) $ and all $(x,t) \in\left( \mathcal{O}_\nu \cap \left\lbrace |x| > 1+2\nu \right\rbrace \right) \times ]0,1]$} we have
\begin{equation}\label{Step2TheoremHeat}
\dfrac{c_1'}{1+|x_0|^{d+2\alpha}} 1_{B(0,\nu)}(x-x_0) \leq \frac{c_1}{1+|x|^{d+2\alpha}} \leq  \partial_t v(x,t) + L^\alpha v(x,t).
\end{equation}
Inequality \eqref{Step2TheoremHeat} gives that for all $(x,t) \in \left( \mathcal{O}_\nu \cap \left\lbrace |x| > 1+\nu \right\rbrace \right) \times ]0,1]$
\begin{equation*}
1_{B(0,\nu)}(x-x_0)  \leq \partial_t (\frac{1+|x_0|^{d+2\alpha}}{c_1'} v(x,t)) + L^{\alpha} (\frac{1+|x_0|^{d+2\alpha}}{c_1'} v(x,t)) .
\end{equation*}
Then, according to the comparison principle and Lemma \ref{dernierlemme}, we deduce that
\textcolor{black}{\begin{equation}\label{equation_avec_a}
\forall x \in \left( \mathcal{O}_\nu \cap \left\lbrace |x| > 1+\nu \right\rbrace \right),  \quad 
c_\nu \widetilde{\phi}_\nu (x-x_0)  \leq \dfrac{1+|x_0|^{d+2\alpha}}{c_1'} v(x,t=1) .
\end{equation}}
If we evaluate \eqref{equation_avec_a} at $x=x_0$, we obtain
\[\dfrac{c_\nu c_1' e^{-\lambda} \widetilde{\phi}_\nu(0)}{1+|x_0|^{d+2\alpha}}  \leq  p( x_0,t=1). \]
Defining $c_2 = c_\nu c_1' e^{-\lambda}  \widetilde{\phi}_\nu(0) $ leads to \eqref{step2heat}.
\bigbreak

\textbf{Step 4. } As in the proof of Proposition \ref{proposition:shape}, we can show by contradiction that there exists a positive constant $c_0$ such that for all $x \in \mathbb{R}^d$, 
\[c_0\underline{\psi}(x)\leq  \widetilde{\phi}_\nu (x)\]
where $\underline{\psi}$ is defined in Lemma \ref{barrier_function_lemma}.
Then we take \textcolor{black}{$x_1 \in \left( \mathcal{O} \backslash \mathcal{O}_\nu \right) \cap \left\lbrace |x|>1+2\nu \right\rbrace$. Since $\mathcal{O}$} satisfies the uniform interior ball condition, there exists $x_0 \in \partial \mathcal{O}_\nu$ such that $x_1 \in B(x_0, \nu)$, $B(x_0, \nu) \subset \mathcal{O} \cap \left\lbrace |x|>1+\nu \right\rbrace$ and $\nu - |x_1-x_0| = \delta(x_1)$. Thanks to \eqref{equation_avec_a} and the fourth point of Lemma \ref{barrier_function_lemma}, we deduce 
\begin{equation*}\label{equation:bord:heat:kernel:2}
 c_\nu c_0 \underline{c}\nu \delta(x_1)^\alpha     \leq c_\nu c_0 \underline{\psi}(x_1 - x_0)  \leq  c_\nu \phi_\nu(x_1-x_0) \leq \dfrac{(|x_0|+ 1)^{d+2\alpha}}{c_1'} v(x_1,t=1) .
\end{equation*}
We deduce that there exists $c_3 >0$ such that \eqref{step3heat} holds true.
\bigbreak
Combining \eqref{step2heat}, \eqref{step3heat} and \eqref{infp} yields the conclusion of the Theorem.
\end{proof}

We apply Theorem \ref{heat} to show that starting from $n(x,0) \in C^\infty_0(\Omega) \cap C^\infty_c(\mathbb{R})$, the solution of \eqref{equation_1} $n( \cdot, t=1)$ has algebraic tails.

\begin{proposition}
There exists two constants $c_m$ and $c_M$ depending on $n_0$ such that for all $x \in \Omega$, we have 
\begin{equation}\label{NewInitialData}
 \frac{c_m \delta(x)^\alpha}{1+|x|^{d+2\alpha}} \leq n(x,1) \leq \frac{c_M \delta(x)^\alpha }{1+|x|^{d+2\alpha}}.
\end{equation}
\end{proposition}

\begin{proof}
Defining $M:= \max(\max n_0,1)$, the solution $n$ belongs to the set $[0,M]$ ($0$ is a sub-solution and $M$ is a super-solution). 
\bigbreak
We begin with the proof that $\frac{c_m \delta(x)^\alpha}{1+|x|^{d+2\alpha}} \leq n(x,1) $. 

Let $\underline{n}$ be the solution of :
\begin{equation}\label{sub equation}
\left\lbrace
\begin{aligned}
&\partial_t \underline{n} (x,t) + (-\Delta)^\alpha \underline{n} (x,t) = -M \underline{n}(x,t) && \text{ for all }(x,t) \in \Omega \times ]0,+\infty[,\\
&\underline{n}(x,t)=0 && \text{ for all } (x,t) \in \Omega^c \times [0,+\infty[, \\
&\underline{n}(x,0)=n_0(x) && \text{ for all } x \in \mathbb{R}^d,
\end{aligned}
\right.
\end{equation}
Thanks to the comparison principle, we deduce that for all $(x,t) \in \mathbb{R} \times [0, +\infty[$, we have 
\[\underline{n}(x,t) \leq n(x,t).\]
Moreover, if we define $\underline{p}(x,t) = e^{Mt} \underline{n}(x,t)$, we find that $\underline{p}$ is solution of \eqref{app_eq_1}. Since $\Omega$ fullfies the uniform interior and exterior ball condition, we deduce thanks to Theorem \ref{heat} that there exists $c_m>0$ such that  
\begin{equation}
 \dfrac{c_m \delta(x)^\alpha}{1+|x|^{d+2\alpha}} \leq \underline{n}(x,t=1) \leq n(x,t=1). 
\end{equation}

The proof works the same for the other bound.
\end{proof}

%In what follows, we make a translation in time to keep $n(x,1)$ as our initial data. In other words, we will suppose that there exists $c_m$ and $c_M$ such that:
%\begin{equation}\label{condition_initial}
%\tag{H2'}
%\forall x \in \mathbb{R}, \ \dfrac{c_m\times\delta(x)^\alpha}{1+|x|^{1+2\alpha}}\leq n(x,0) \leq \dfrac{c_M}{1+|x|^{1+2\alpha}}.
%\end{equation}

\section{The proof of Theorem \ref{theorem_convergence_1}}

\subsection{Rescaling and preparation}
\textcolor{black}{The aim of this subsection is to establish the following Theorem. 
\begin{theorem}\label{Cor2}
We assume \eqref{H1} and \eqref{H2} then for all \textcolor{black}{$\nu>0$}, the following holds true
\begin{enumerate}
\item For all $c < \frac{|\lambda_0|}{d+2\alpha}$, there exists a constant $\sigma>0$ and a time $t_{\sigma}>0$ such that 
\begin{equation}\label{step1theorem4}
\forall (x,t) \in \left( \Omega_\nu \times \left\lbrace | x|  < e^{ct} \right\rbrace \right) \times ]t_\sigma, +\infty[ \text{ we have } \sigma < n(x,t)  .
\end{equation}
\textcolor{black}{\item For all $C>\frac{|\lambda_0|}{d+2\alpha}$, there exists two constants $\overline{C},\kappa>0$ such that we have for all $(x,t) \in \left\lbrace |x|>e^{Ct} \right\rbrace \times ]1, +\infty[$
\begin{equation}\label{eqCor3}
n(x,t) \leq \frac{\overline{C}}{1+e^{\kappa t } }. 
\end{equation}}
\end{enumerate}
\end{theorem}}
First we establish sub and super-solutions by performing the rescaling \eqref{scale}. Finally, we prove Theorem \ref{Cor2} by performing the inverse \textcolor{black}{of this rescaling} on the sub and super-solutions. %convergence results on the principal eigenvalue $\lambda_\nu$ of $(-\Delta)^\alpha - Id$ in $\Omega_\nu$ and on $n_{+, \nu}$ the stationary state of \eqref{equation_1_stat_1} in $\Omega_\nu$ when $\nu$ is small. The main argument to obtain those results is Theorem \ref{theorem_existence_uniqueness_stationnary_state}. 

%\subsection{The scaling and the sub and super-solutions}

%We also introduce the following set :
%\[\mathcal{B}^{\delta} = \left\lbrace (x,t) \in \mathbb{R} \times [0,+\infty[\  | \  |x|^{1+2\alpha } \leq e^{|\lambda_0|t-\delta} \right\rbrace.\] 
We rescale the solution of \eqref{equation_1} as follows : 
\begin{equation}\label{def_n_varepsilon}
n_\varepsilon (x,t) = n\left(|x|^{\frac{1}{\varepsilon}}\dfrac{x}{|x|}, \dfrac{t}{\varepsilon}\right).
\end{equation}
Next, the equation becomes 
\begin{equation}
\tag{$1_\varepsilon$}
\label{equation_scall}
\left\lbrace
\begin{aligned}
&\varepsilon \partial_t n_\varepsilon + (-\Delta)^\alpha_\varepsilon n_\varepsilon = n_\varepsilon(1-n_\varepsilon) \ && \text{ for } (x,t) \in \Omega^\varepsilon \times ]0,+\infty[, \\
&n_\varepsilon(x,t) =0 \ && \text{ for } (x,t) \in \Omega^{\varepsilon^c} \times ]0,+\infty[,\\
&n_\varepsilon(x,0)=n_{0,\varepsilon}(x) &&\in C^{\infty}_0(\Omega^\varepsilon, \mathbb{R}^+). \\
\end{aligned}
\right.
\end{equation}
where $(-\Delta)^\alpha_\varepsilon n_\varepsilon (x,t)= (-\Delta)^\alpha n\left( |x|^\frac{1}{\varepsilon}\frac{x}{|x|}, \frac{t}{\varepsilon}\right)$ and $\Omega^\varepsilon = \left\lbrace x \in \mathbb{R}^d \ | \ |x|^{\frac{1}{\varepsilon}-1}x \in \Omega \right\rbrace$.\\
Next, we set
\[g(x) := \dfrac{1}{1+|x|^{d+2\alpha}}.\]
We state the behavior of $g$ under the fractional Laplacian in the spirit of \cite{Cabre-Coulon-Roquejoffre}.  

\begin{lemma}
\label{lemme_utile}
Let $\gamma$ be a positive constant in $]0, \alpha[$ such that $2\alpha - \gamma < 1$. Let $\chi \in C^\alpha(\mathbb{R}^d)$ be a periodic positive function. Then there exists a positive constant $C$, such that we have for all $x \in \mathbb{R}$: \\
(i) for all $a>0$,
\[| (-\Delta)^{\alpha} g (ax) | \leq a^{2 \alpha} C g(ax), \]
(ii) for all $a \in ]0,1]$, 
\[|\widetilde{K}(g(a .),\chi)(x) | \leq \frac{C a ^{2 \alpha - \gamma}}{1+(a |x|)^{d+2\alpha}} = C a^{2 \alpha - \gamma} g(a |x|),\]
where $\widetilde{K}(u,v)(x)=\int_{\mathbb{R}} \frac{(u(x)-u(y))(v(x)-v(y))}{|x-y|^{d+2\alpha}}dy$ is such that 
\[(-\Delta)^\alpha (u\times v)= (-\Delta)^\alpha (u )\times v + u \times (-\Delta)^\alpha (v)- \widetilde{K}(u,v).\]
\end{lemma}

Since, the same kind of result can be found in the appendix A of \cite{Leculier_1}, we do not provide the proof of this lemma. Note that here, the lemma is stated with less regularity on $\chi$ such than in \cite{Leculier_1}. Nevertheless, there is no difficulty to adapt the proof.

\begin{notation}
As we have introduced $(-\Delta)^\alpha_\varepsilon n_\varepsilon (x,t)= (-\Delta)^\alpha n\left( |x|^\frac{1}{\varepsilon}\frac{x}{|x|}, \frac{t}{\varepsilon}\right)$, we introduce 
\[\widetilde{K}_\varepsilon (u,v)(x,t)= \widetilde{K}(u,v)( |x|^{\frac{1}{\varepsilon}-1}, \frac{t}{\varepsilon} ).\]
For any application $h : \mathbb{R}\mapsto \mathbb{R}$, we define 
\begin{align*}
h_\varepsilon : \ & \mathbb{R}^d \rightarrow \mathbb{R}\\
&x \mapsto h(|x|^{\frac{1}{\varepsilon}-1} x).
\end{align*}
For any set $\mathcal{U}$, we will denote
\begin{equation}
 \mathcal{U}^\varepsilon  =  \left\lbrace x \in \mathbb{R}^d | \  |x|^{\frac{1}{\varepsilon}-1}x \in \mathcal{U} \right\rbrace .
\end{equation}
For reasons of brevity, we will always denote $\left( \mathcal{U}_\nu \right)^\varepsilon$ by $\mathcal{U}_\nu^\varepsilon$. 
\end{notation}

\textcolor{black}{In the following, we denote by $\lambda_\nu$ the principal eigenvalue of the Dirichlet operator $(-\Delta)^\alpha - Id$ in $\Omega_\nu$ and $\phi_\nu$ the associated eigenfunction which can be chosen positve. Then, the following proposition hold true.
\begin{proposition}\label{prop_convergence_lambda_nu}
The map $\left( \nu \in ]-r_0, r_0[ \mapsto \lambda_\nu \right)$ is increasing and continuous.
\end{proposition}
 We do not provide the proof since there is no difficulty: it relies on the Rayleigh quotient for the monotonicity and on the uniqueness of the principal eigenvalue for the continuity. We deduce thanks to \eqref{H1} and Proposition \ref{prop_convergence_lambda_nu} that
\begin{equation}\label{r0}
 \exists r_0 > 0, \text{ such that for all }  \nu \in ]0,r_0[, \quad \lambda_{\nu}<0.
\end{equation} 
Next using the eigenfunctions $\phi_{\pm\nu}$ we establish the sub and super-solution to \eqref{equation_scall}.}

\begin{proposition}\label{sub_sup}
We assume \eqref{H1} and \eqref{H2}. Let $\nu$ be a positive constant such that $\nu< r_0$. If we set $C_m = \min(\frac{\min(|\lambda_{\nu}|, c_m \nu^\alpha)}{2 (\max \  \phi_\nu+1)}, 1)$ and $C_M = \frac{2|\lambda_{-\nu}| + c_M}{\underset{\Omega}{\min}\ \phi_{-\nu}}$ %and $\delta=\min(\sqrt{|\lambda_{\nu}|-C_m \max\  \phi_\nu},\sqrt{C_M \underset{\Omega}{\min} \ \phi_{-\nu}-|\lambda_{-\nu}|})$. 
then there exists $\varepsilon_\nu > 0$ such that for all $\varepsilon<\varepsilon_\nu$, the following holds true.
\begin{enumerate}%[label=(\roman*)]
\item  If $f_\varepsilon^m$ is defined as 
\textcolor{black}{\[ f_\varepsilon^m(x,t)= \dfrac{C_m \min(e^{-\frac{1}{\varepsilon}+\frac{\varepsilon t}{4}},1) }{1+e^{-\frac{t}{\varepsilon}(|\lambda_{\nu}|-\varepsilon^2)-\frac{1}{\varepsilon}}|x|^{\frac{d+2\alpha}{\varepsilon}}}  \times (\phi_{\nu,\varepsilon} (x) + \varepsilon)  ,\]}%\times 1_{\Omega_{\nu}^\varepsilon}(x),\]
then it is a sub-solution of \eqref{equation_scall} in $\Omega_\nu^\varepsilon \times \left( ]0,\frac{4}{\varepsilon^2}[ \cup ]\frac{4}{\varepsilon^2}, +\infty[ \right)$.
\item If $f_\varepsilon^M $ is defined as
\[  f_\varepsilon^M(x,t)= \frac{C_M \times \phi_{-\nu,\varepsilon}(x) }{1+e^{\frac{-t}{\varepsilon}(|\lambda_{-\nu}|+\varepsilon^2)-\frac{1}{\varepsilon}}|x|^\frac{d+2\alpha}{\varepsilon}} , \]
then it is a super-solution of \eqref{equation_scall} in $\Omega^\varepsilon \times ]0, +\infty[$.
\item For all $ (x,t) \in \mathbb{R}^d \times [0,+\infty[, \ \ f_\varepsilon^m (x,t) \leq n_\varepsilon(x,t + \varepsilon) + \varepsilon$ and $n_\varepsilon (x,t+ \varepsilon) \leq f_\varepsilon^M (x,t )  .$
\end{enumerate}
\end{proposition}

\begin{proof}
We begin by proving \textit{(1)}. We split the study into two parts : when $t \in ]0, \frac{4}{\varepsilon^2}[$ and $t \in ]\frac{4}{\varepsilon^2}, +\infty[$. Let $(x,t)$ be in $\Omega_\nu^\varepsilon  \times ]0, \frac{4}{\varepsilon^2}[$. We define:
\begin{align*}
\psi_\varepsilon (x,t)=&\dfrac{C_m}{1+e^{-\frac{t}{\varepsilon}(|\lambda_{\nu}|-\varepsilon^2)-\frac{1}{\varepsilon}}|x|^\frac{d+2\alpha}{\varepsilon}} = C_m g_\varepsilon(e^{-\frac{t(|\lambda_{\nu}|-\varepsilon^2)-1}{d+2\alpha}}x) \quad \text{ and } \quad a(t) = e^{-\frac{1}{\varepsilon} + \frac{\varepsilon t}{4}}\\
\qquad \text{ thus }	&\qquad f_\varepsilon^m(x,t) =a(t) \times \psi_\varepsilon (x,t) \times(\phi_{\nu,\varepsilon} (x) + \varepsilon).
\end{align*}
First, we bound $\varepsilon \partial_t  \psi_\varepsilon$ from above:
\begin{equation}\label{derive_temporelle}
\begin{aligned}
\varepsilon \partial_t \psi_\varepsilon (x,t)& =   \varepsilon \dfrac{C_m \frac{(|\lambda_{\nu}|-\varepsilon^2)}{\varepsilon}e^{-\frac{t}{\varepsilon}(|\lambda_{\nu}|-\varepsilon^2)-\frac{1}{\varepsilon}}|x|^{\frac{d+2\alpha}{\varepsilon}}}{(1+e^{-\frac{t}{\varepsilon}(|\lambda_{\nu}|-\varepsilon^2)-\frac{1}{\varepsilon}}|x|^{\frac{d+2\alpha}{\varepsilon}})^2} \\
&= \psi_\varepsilon (x,t) \left[(|\lambda_{\nu}|- \varepsilon^2)\frac{e^{-\frac{t}{\varepsilon}(|\lambda_{\nu}|-\varepsilon^2)-\frac{1}{\varepsilon}}|x|^{\frac{d+2\alpha}{\varepsilon}}}{1+e^{-\frac{t}{\varepsilon}(|\lambda_{\nu}|-\varepsilon^2)-\frac{1}{\varepsilon}}|x|^{\frac{d+2\alpha}{\varepsilon}}} \right] \\
& \leq \psi_\varepsilon (x,t) [|\lambda_{\nu}| - \varepsilon^2 - \psi_\varepsilon(x,t) (\phi_{\nu,\varepsilon}(x)+\varepsilon)] \\
& \leq \psi_\varepsilon (x,t) [|\lambda_{\nu}| - \varepsilon^2 - f_\varepsilon(x,t)] .
\end{aligned}
\end{equation}
The last inequalities hold because $a(t) \leq 1$ and denoting by $D = e^{-\frac{t}{\varepsilon}(|\lambda_{\nu}|-\varepsilon^2)-\frac{1}{\varepsilon}}|x|^\frac{d+2\alpha}{\varepsilon}$ and using the definition of $C_m$, we obtain for all $\varepsilon < \min(\sqrt{\frac{|\lambda_\nu|}{2}},1)$
\begin{align*}
|\lambda_\nu| - \varepsilon^2-\psi_\varepsilon (\phi_{\nu,\varepsilon} + \varepsilon) - \left( |\lambda_\nu| - \varepsilon^2\right) \frac{D}{1+D} &=\frac{|\lambda_{\nu}| - \varepsilon^2-C_m (\phi_{\nu, \varepsilon} + \varepsilon)}{1+D}\\
&\geq \frac{|\lambda_\nu| - \varepsilon^2 - \frac{|\lambda_\nu|}{2}}{1+D}\\
&\geq 0.
\end{align*}
Next, we compute $(-\Delta)^\alpha_\varepsilon f^m_\varepsilon (x,t)$ 
\begin{equation*}   
(-\Delta)^\alpha_{\varepsilon} f^m_\varepsilon (x,t) =a(t) (\phi_{\nu,\varepsilon} + \varepsilon)(x)(-\Delta)^\alpha_{\varepsilon} \psi_\varepsilon(x,t) +  a(t) \psi_\varepsilon(x,t)(-\Delta)^\alpha_{\varepsilon} \phi_{\nu,\varepsilon}(x)  - a(t) \widetilde{K}_\varepsilon(\psi, (\phi_{\nu}+\varepsilon) )(x,t).
\end{equation*}
Combining \eqref{derive_temporelle} and the above equality we find:

\begin{equation}\label{in_1_f}
\begin{aligned}
&\varepsilon \partial_t f^m_\varepsilon (x,t) + (-\Delta)^\alpha_{\varepsilon} f^m_\varepsilon(x,t) - f^m_\varepsilon(x,t)(1 - f^{m}_{\varepsilon}(x,t)) \\
&\leq \textcolor{black}{ f^m_\varepsilon(x,t) (|\lambda_{\nu}|-\frac{3\varepsilon^2}{4} - f^m_\varepsilon(x,t)) } +a(t) (\phi_{\nu,\varepsilon}(x) + \varepsilon ) (-\Delta)^\alpha_\varepsilon \psi_\varepsilon(x,t) +  a(t)\psi_\varepsilon(x,t)(-\Delta)^\alpha_{\varepsilon} \phi_{\nu,\varepsilon}(x)\\
&-a(t)\widetilde{K}_\varepsilon(\psi, \phi_{\nu} +\varepsilon)(x,t) - f^m_\varepsilon(x,t)(1-f^m_\varepsilon(x,t)) \\
&= f^m_\varepsilon (x,t)(|\lambda_{\nu}|-\frac{3\varepsilon^2 }{4} ) + a(t)(\phi_{\nu,\varepsilon}(x) + \varepsilon) (-\Delta)^\alpha_\varepsilon \psi_\varepsilon(x,t) + a(t) (\lambda_\nu +1)\psi_\varepsilon(x,t) \phi_{\nu,\varepsilon}(x)  \\
&- f_\varepsilon^m(x,t) -a(t)\widetilde{K}_\varepsilon(\psi, \phi_{\nu} +\varepsilon )(x,t) \\
&= f^m_\varepsilon(x,t) (|\lambda_{\nu}|-\frac{3\varepsilon^2 }{4}) + a(t)(\phi_{\nu,\varepsilon}(x) + \varepsilon) (-\Delta)^\alpha_\varepsilon \psi_\varepsilon(x,t) +  a(t)(\lambda_\nu +1)\psi_\varepsilon(x,t) (\phi_{\nu, \varepsilon}(x)+\varepsilon) (x,t) \\
&  - a(t)\varepsilon (\lambda_\nu +1)\psi_\varepsilon(x,t)  - f_\varepsilon^m (x,t)-a(t)\widetilde{K}_\varepsilon(\psi, \phi_{\nu} +\varepsilon)(x,t) \\
&= -\frac{3\varepsilon^2}{4}  f^m_\varepsilon  (x,t)+ a(t)(\phi_{\nu,\varepsilon}(x) + \varepsilon) (-\Delta)^\alpha_\varepsilon \psi_\varepsilon(x,t)  -a(t) \varepsilon (\lambda_\nu +1)\psi_\varepsilon(x,t)  -a(t) \widetilde{K}_\varepsilon(\psi, \phi_{\nu}+\varepsilon)(x,t).
\end{aligned}
\end{equation}
Thanks to Lemma \ref{lemme_utile}, we obtain
\begin{align*}
|(-\Delta)^\alpha_\varepsilon \psi_\varepsilon(x,t)|&=|C_m  (-\Delta)_\varepsilon^\alpha (g_\varepsilon(e^{\frac{-[t(|\lambda_\nu|-\varepsilon^2)+1]}{d+2\alpha}} .))(x)|\\
& \leq |C_m e^{-\frac{2\alpha [t(|\lambda_\nu|-\varepsilon^2)+1]}{\varepsilon(d+2\alpha)}}  (g_\varepsilon)(e^{\frac{-[t(|\lambda_\nu|-\varepsilon^2)+1]}{d+2\alpha} } x)|.
\end{align*}
We deduce that there exists $\varepsilon_1 > 0$ such that for all $\varepsilon< \varepsilon_1$:
\begin{equation}\label{probleme_1}
|(-\Delta)^\alpha_\varepsilon \psi_\varepsilon(x,t)| \leq \frac{\varepsilon^2}{4} \psi_\varepsilon (x,t).
\end{equation}
Since $(\phi_\nu+\varepsilon)$ is periodic, positive and $C^\alpha$ according to \cite{Ros-Oton-Serra} (Proposition 1.1), we conclude from Lemma \ref{lemme_utile} that there exists $\gamma \in ]0, \alpha[$ and a constant $C$ such that
\[|\widetilde{K}_\varepsilon (\psi,\phi_\nu+\varepsilon)(x,t)|\leq C e^{-\frac{[t(|\lambda_\nu| - \varepsilon^2 )+1](2\alpha-\gamma)}{\varepsilon(d+2\alpha)}} \psi_\varepsilon (x,t).\]
We deduce the existence of $\varepsilon_2>0$ such that for all $\varepsilon<\varepsilon_2$, we have
\begin{equation}\label{probleme_2}
|\widetilde{K}_\varepsilon (\psi,\phi_\nu+\varepsilon)(x,t)|\leq \frac{\varepsilon^3}{4}  \psi_\varepsilon(x,t) = \frac{ \varepsilon^2\min (\phi_{\nu, \varepsilon } + \varepsilon)}{4} \psi_\varepsilon(x,t) .
\end{equation}
Noticing that $(\lambda_\nu +1)>0$, inserting \eqref{probleme_1} and \eqref{probleme_2} into \eqref{in_1_f}, we conclude that for all $\varepsilon< \varepsilon_\nu:= \min (\varepsilon_1, \varepsilon_2, \sqrt{ \frac{|\lambda_\nu|}{2}},1)$ and $(x,t) \in \Omega_\nu^\varepsilon\times [0,+\infty[$ we have:
\begin{align*}
&\varepsilon \partial_t f^m_\varepsilon (x,t) + (-\Delta)^\alpha_{\varepsilon } f^m_\varepsilon (x,t)- f^m_\varepsilon (x,t)+ {f^m_\varepsilon}(x,t)^2   \\
& \leq -\frac{3\varepsilon^2}{4} f_\varepsilon^m(x,t) + a(t)(\phi_{\nu, \varepsilon}+\mu) (x)  (-\Delta)^\alpha \psi_\varepsilon(x,t) - a(t)\varepsilon(\lambda_\nu +1) \psi_\varepsilon(x,t)  - a(t)\widetilde{K}_\varepsilon (\psi, \phi_\nu + \varepsilon) (x,t) \\
& \leq -\frac{3\varepsilon^2}{4} f_\varepsilon^m(x,t) + \frac{\varepsilon^2}{4} f_\varepsilon(x,t) + \frac{\varepsilon^2}{4} f_\varepsilon(x,t)  \\
&\leq -\frac{\varepsilon^2}{4} f^m_\varepsilon (x,t) \\
& \leq 0.
\end{align*}
Therefore, $f_\varepsilon^m$ is a sub-solution of \eqref{equation_scall} for $(x,t) \in \Omega_\nu^\varepsilon \times ]0, \frac{4}{\varepsilon^2}[$.

\bigbreak

We conclude the proof of \textit{(1)} with the same computations where we replace $a(t)$ by $1$. It turns out that for all $(x,t) \in \Omega_\nu^\varepsilon \times]\frac{4}{\varepsilon^2}, +\infty[$ we have 
\[\varepsilon \partial_t f^m_\varepsilon (x,t) + (-\Delta)^\alpha_{\varepsilon } f^m_\varepsilon (x,t)- f^m_\varepsilon (x,t)+ {f^m_\varepsilon}(x,t)^2  \leq -\frac{\varepsilon^2}{2} f^m_\varepsilon (x,t) < 0.\] 

\bigbreak

The proof of \textit{(2)} follows the same arguments as the proof of \textit{(1)}.

\bigbreak

For the proof of \textit{(3)}, we have to check that the initial data are ordered in the right way. According to \eqref{NewInitialData} and the definition of $C_m$, we have that for all $x \in \Omega_\nu^\varepsilon$, 
\[ f_\varepsilon^m (x,0) = \frac{C_m (\phi_{\nu,\varepsilon}(x) + \varepsilon)}{e^\frac{1}{\varepsilon} + |x|^{\frac{d+2\alpha}{\varepsilon}}} \leq \frac{C_m (\max \phi_\nu + \varepsilon)}{1+|x|^{\frac{d+2\alpha}{\varepsilon}}} \leq \frac{ c_m \nu^{\alpha} + \varepsilon }{1+|x|^{\frac{d+2\alpha}{\varepsilon}}} \leq \frac{ c_m \delta(x)^{\alpha}}{1+|x|^{\frac{d+2\alpha}{\varepsilon}}} + \varepsilon\leq n_\varepsilon(x, \varepsilon) +\varepsilon .\]
Furthermore,
\begin{equation}\label{eqth4}
\forall (x,t) \in (\Omega_\nu^\varepsilon)^c\times [0,+\infty[, \text{  we know that }  f_\varepsilon^m(x,t) \leq \varepsilon \leq n_\varepsilon(x,t+ \varepsilon) + \varepsilon.
\end{equation}
Thus we conclude from the comparison principle that for all $(x,t) \in \mathbb{R}^d \times [0,\frac{4}{\varepsilon^2}]$, we have 
\begin{equation} 
f_\varepsilon^m (x,t) \leq n_\varepsilon(x,t + \varepsilon) + \varepsilon.
\end{equation}
\textcolor{black}{Since, we have that for all $x \in \mathbb{R}^d$  
\[f^m_\varepsilon(x, \frac{4}{\varepsilon^2}) \leq n_\varepsilon (x, \frac{4}{\varepsilon^2} + \varepsilon) + \varepsilon\]
and recalling that $f^m_\varepsilon$ is also a subsolution in $\Omega^\varepsilon_\nu \times ]\frac{4}{\varepsilon^2}, +\infty[$ and the inequality \eqref{eqth4}, we deduce thanks to the comparison principle that for all $(x,t) \in \mathbb{R}^d \times [0,+\infty[$}
\begin{equation}\label{sub_solution}  
f_\varepsilon^m (x,t) \leq n_\varepsilon(x,t + \varepsilon) + \varepsilon.
\end{equation}

\bigbreak

The other inequality can be obtained following similar arguments.
\end{proof}

A direct consequence of \eqref{sub_solution} is that if $\varepsilon$ fulfills the assumption of Proposition \ref{sub_sup} then 
\begin{equation}\label{sub_solution2}
\forall (x,t) \in \mathbb{R}^d \times ]\frac{4}{\varepsilon^2}, +\infty[ \quad \dfrac{C_m \times \phi_{\nu,\varepsilon} (x)}{1+e^{-\frac{t}{\varepsilon}(|\lambda_{\nu}|-\varepsilon^2)}|x|^{\frac{d+2\alpha}{\varepsilon}}}   \leq n_\varepsilon (x , t + \varepsilon) + \varepsilon.
\end{equation}

Next, we establish some consequences of Theorem \ref{sub_sup} on the solution $n$ without the scaling \eqref{scale}.

\begin{proof}[Proof of Theorem \ref{Cor2}]
\textcolor{black}{First, we prove the first point by using the sub-solution $f_\varepsilon^m$. \textcolor{black}{It is sufficient to prove it for $\nu < r_0$ (where $r_0$ is introduced in \eqref{r0}).}\\
\textit{Proof of 1. } Set $\nu \in ]0, r_0[$ and $c < \frac{|\lambda_0|}{d+2\alpha}$. According to Proposition \ref{prop_convergence_lambda_nu}, there exists two positive constants $\nu_0 < \nu$ and $\varepsilon_0>0$ such that, \textcolor{black}{for all $\varepsilon < \varepsilon_0$}
\begin{equation}\label{pfCorollary1}
(d+2\alpha)c-|\lambda_\frac{\nu_0}{2}| + \varepsilon^2<0.
\end{equation}
Moreover by Proposition \ref{sub_sup}, \textcolor{black}{for all $\nu \in ]0, r_0[$}, there exists $\varepsilon_\nu>0$ such that for all $\varepsilon< \varepsilon_\nu$ and all  $(x,t) \in  \mathbb{R} \times ]\frac{4}{\varepsilon^2}, +\infty[$, \eqref{sub_solution} holds true. Therefore, for $\varepsilon = \frac{\min \left(\varepsilon_0, \varepsilon_\frac{\nu_0}{2}, \frac{C_m \underset{\Omega_{\nu_0}}{\min} \phi_\frac{\nu_0}{2} }{4} \right)}{2}$ we deduce thanks to \eqref{sub_solution2} that for all $(x,t) \in \left( \Omega_{\nu}^\varepsilon \cap \left\lbrace |x| < e^{ct} \right\rbrace \right) \times ]\frac{4}{\varepsilon^2}, +\infty[$ we have 
%\left\lbrace (x,t) \in \mathbb{R}^d \times ]0 , +\infty[ | \ | x|  < e^{\frac{(|\lambda_0|-\delta)t}{1+2\alpha}} \text{ and } |x|^{\frac{1}{\varepsilon}-1}x \in \Omega_\delta \right\rbrace$ 
\begin{equation*}
  \frac{C_m  \phi_\frac{\nu_0}{2}(|x|^{\frac{1}{\varepsilon} - 1}x) }{1+e^{\frac{-t(|\lambda_\frac{\nu_0}{2}|-\varepsilon^2)}{\varepsilon}}|x|^{\frac{d+2\alpha}{\varepsilon}}} \leq  n_\varepsilon(x,t + \varepsilon) + \varepsilon \quad \Rightarrow \quad   \frac{C_m \  \underset{\Omega_{\nu_0}}{\min} \ \phi_\frac{\nu_0}{2}}{1+e^{\frac{t}{\varepsilon}(c-|\lambda_\frac{\nu_0}{2}|+\varepsilon^2)}} \leq n_\varepsilon(x,t+ \varepsilon) + \varepsilon. 
\end{equation*}
If we perform the inverse scaling to \eqref{scale}, it follows thanks to \eqref{pfCorollary1} that for all $(x,t) \in \textcolor{black}{\left(\Omega_\nu \cap \left\lbrace |x| < e^{ct} \right\rbrace \right)} \times ]\frac{4}{\varepsilon^3} + 1, +\infty[$
\[\frac{C_m  \ \underset{\Omega_{\nu_0}}{\min}\  \phi_\frac{\nu_0}{2} }{2} \leq \frac{C_m \  \underset{\Omega_{\nu_0}}{\min} \  \phi_\frac{\nu_0}{2} }{1+e^{t(c-|\lambda_\frac{\nu_0}{2}|+\varepsilon^2)}} \leq n(x,t+1) + \varepsilon.\]
If we define $\sigma = \frac{C_m \  \underset{\Omega_{\nu_0}}{\min} \  \phi_\frac{\nu_0}{2} }{4}$ and $t_{\sigma}  = \frac{4}{\varepsilon^3}+1$, we conclude that \eqref{step1theorem4} holds true.}

\bigbreak

\textcolor{black}{We prove the second point by using the super-solution $f^M_\varepsilon$. \\
\textit{Proof of 2. } Let $C>\frac{|\lambda_0|}{d+2\alpha}$.  According to Proposition \ref{prop_convergence_lambda_nu}, we deduce the existence of $\nu, \varepsilon > 0$ such that 
\[ |\lambda_0 | < |\lambda_{-\nu} | + \varepsilon^2 < (d+2\alpha)C \quad \text{ and } \quad \text{ Proposition \ref{sub_sup} holds true}.\]
Proposition \ref{sub_sup} implies that 
\[\forall (x,t) \in \mathbb{R}^d \times ]0, +\infty[, \quad n_\varepsilon (x ,t+ \varepsilon) \leq f^M_\varepsilon (x,t).\]
If we perform the  scaling $\left( (x,t) \mapsto (|x|^{\varepsilon -1} x, \varepsilon t) \right)$, it follows that 
\[n(x,t+1) \leq \frac{C_M \times \max \phi_{-\nu}}{1+e^{-t(|\lambda_{-\nu}|+\varepsilon^2) - \frac{1}{\varepsilon}}|x|^{d+2\alpha}}.\]
Then for all $(x,t) \in \left\lbrace |x|>Ct \right\rbrace \times ]0, +\infty[$ we have 
\[n(x,t + 1 ) \leq \frac{C_M \times \max \phi_{-\nu}}{1+e^{t[(d+2\alpha)C-(|\lambda_{-\nu}|+\varepsilon^2)] - \frac{1}{\varepsilon}}}.\]
Defining $\overline{C}:= 2C_Me^{\frac{1}{\varepsilon}} \max \phi_{-\nu}$ and $\kappa :=(d+2\alpha)C-(|\lambda_{-\nu}|+\varepsilon^2)$ then the conclusions follows. 
}
\end{proof}

\subsection{The final argument}

\begin{proof}[Proof of Theorem \ref{theorem_convergence_1}]

\textcolor{black}{We will prove $(i)$ by splitting the proof into two parts : the upper bound and the lower bound. We will not provide the proof of $(ii)$ since it is a direct application of \textit{2.} of Theorem \ref{Cor2}. }

\bigbreak

\textcolor{black}{\textit{Proof of $(i)$.} Let $\mu $ be a positive constant. We want to prove that there exists a time $t_\mu>0$ such that for any $c < \frac{|\lambda_0|}{d+2\alpha}$ we have for all $(x,t) \in  \left\lbrace |x| < e^{ct} \right\rbrace \times ]t_\mu, +\infty[$ 
\[ |n(x,t) - n_+(x) | \leq \mu.\]
First we establish that there exists a time $t_1>0$ such that 
\begin{equation}\label{pf3obj1}
\forall (x,t) \in \Omega \times ]t_1, +\infty[, \quad n(x,t) - n_+(x) \leq \mu 
\end{equation}
Next, we prove the existence of a time $t_2>0$ such that 
\begin{equation}\label{pf3obj2}
\forall (x,t) \in \Omega \times ]t_2, +\infty[, \quad-\mu \leq n(x,t) - n_+(x)
\end{equation}
The difficult part will be to establish \eqref{pf3obj2}. This is why, we do not provide all the details of the proof of \eqref{pf3obj1}. }\\

\textcolor{black}{\textbf{Proof that \eqref{pf3obj1} holds true. } Thanks to \eqref{NewInitialData} and Proposition \ref{proposition:shape}, we deduce the existence of a constant $C\geq 1$ such that 
\[n(x,t=1) \leq Cn_+(x).\]
Moreover, the solution $\overline{n}$ of 
\begin{equation*}
\left\lbrace
\begin{aligned}
&\partial_t \overline{n} + (-\Delta)^\alpha \overline{n}= \overline{n}- \overline{n} && \text{ in } \Omega \times ]1, +\infty[, \\
& \overline{n}(x,t) = 0 && \text{ in }  \Omega^c \times [1, +\infty[, \\
&\overline{n}(x, t=1) = Cn_+(x) && \text{ in }  \Omega
\end{aligned}
\right.
\end{equation*}
is a super solution of \eqref{equation_1}. According to the comparison principle we deduce that 
\begin{equation}\label{inpf3}
\forall (x,t) \in \mathbb{R} \times [1, +\infty[, \quad n(x,t) \leq \overline{n}(x,t).
\end{equation}
One can easily observe that $\overline{n}$ is periodic, decreasing in time and converges uniformly to $n_+$ in the whole domain $\Omega$ as $t \rightarrow +\infty$. Thus there exists a times $t_1>1$ such that 
\[\forall (x,t) \in \Omega \times ]t_1, +\infty[, \quad \overline{n}(x,t) - n_+(x) \leq \mu.\]
The conclusion follows. }

\textcolor{black}{\textbf{Proof that \eqref{pf3obj2} holds true. } We split this part of the proof into two subparts, what happens on the boundary and what happens in the interior.\\
\textit{The boundary estimates. } Since $\overline{n}$ is decreasing in time and thanks to \eqref{inpf3}, we deduce that for all $(x,t) \in \Omega \times [1, +\infty[$
\[|n(x,t) - n_+(x) | \leq n(x,t) + n_+(x) \leq \overline{n}(x,t) + n_+(x) \leq (C + 1) n_+(x).\]
According to Proposition \ref{proposition:shape}, we deduce that for all $(x,t) \in \Omega \times [1,+\infty[$
\[|n(x,t) - n_+(x) | \leq C(C+1) \delta(x)^\alpha.\]
We conclude that for all $(x,t) \in \Omega \times [1,+\infty[$ such that $\delta(x) < \left( \frac{\mu}{C(C+1)} \right)^\frac{1}{\alpha} := \nu_1$ we have 
\[|n(x,t) - n_+(x) | \leq \mu. \]}

\textcolor{black}{\textit{The interior estimates. } Thanks to Theorem \ref{theorem_existence_uniqueness_stationnary_state}, we know that $n_+ \leq 1$ thus it is sufficient to prove the existence of $t_2>0$ such that 
\[\forall (x,t) \in \left( \left\lbrace |x|< e^{ct}  \right\rbrace \cap \Omega_{\nu_2} \right) \times ]t_2, +\infty[ \quad 1-\mu \leq \frac{n(x,t)}{n_+(x)} \qquad \text{ where } \nu_2 = \min(\nu_1, r_0,r_1) \]
\textcolor{black}{where $\nu_1$ is provided by the previous step, $r_0$ by \eqref{r0} and $r_1$ by the uniform interior ball condition.}
}

\bigbreak

\textcolor{black}{The idea is to approximate $n_+$ by the solution of \eqref{equation_1_stat_1} on a ball of radius $M$. Noticing that thanks to \eqref{H1}, there exists $M_0>0$ such that for $M > M_0$, there exists a unique bounded positive solution $n_{M,+}$ of 
\begin{equation}\label{nM+}
\left\lbrace
\begin{aligned}
(-\Delta)^\alpha n_{M,+}&= n_{M,+}- n_{M,+}^2 && \text{ in } \Omega \cap B(0,M),\\
n_{M,+} &= 0 && \text{ in } (\Omega \cap B(0,M))^c
\end{aligned}
\right.
\end{equation} 
We claim that
\begin{equation}\label{tmu2}
\exists M_1>M_0, \text{ such that } \forall M > M_1, \ \forall x \in\Omega_{0,\nu_2}, \ \left(1-\mu\right)^{\frac{1}{2}} \leq \frac{ n_{M,+}(x)}{n_+ (x) }.
\end{equation}
The proof of this claim is postponed to the end of this paragraph. Next, we approach $n_{M,+}$ by the long time solution of the following equation:
\begin{equation}\label{nM}
\left\lbrace
\begin{aligned}
\partial_t n_{M,z} + (-\Delta)^\alpha n_{M,z}&= n_{M,z}- n_{M,z}^2 && \text{ in } \left(\Omega \cap B(0,M) \right) \times ]0, +\infty[,\\
n_{M,z}(x,t) &= 0 && \text{ in }  (\Omega \cap B(0,M))^c \times ]0, +\infty[,\\
n_{M,z}(x,t=0)&= \sigma 1_{B(z,\frac{\nu_2}{4})}(x).
\end{aligned}
\right.
\end{equation}
where $\sigma$ is provided by Theorem \ref{Cor2} and $z \in \Omega_{0, \frac{\nu_2}{2}}$ will be fixed later on. We claim that 
\begin{equation}\label{t_mu}
\exists t_\mu > 0, \text{ such that } \forall z \in \Omega_{0, \frac{\nu_2}{2}}, \ \forall (x,t) \in \Omega_{0, \nu_2} \times ] t_\mu, +\infty[ , \quad (1-\mu)^\frac{1}{2} \leq \frac{n_{M,z}(x,t)}{n_{M,+}(x)}.
\end{equation}
Again, the proof of this claim is postponed to the end of this section.}
\textcolor{black}{Next, we define
\begin{equation}\label{definition_t_mu}
\underline{t}_\mu = t_\mu + t_\sigma
\end{equation}
where $t_\mu$ is defined by \eqref{t_mu} and $t_\sigma$ by Theorem \ref{Cor2}. Let $(x,t)$ be any couple of $\left( \Omega_{\nu_2} \cap \left\lbrace |x| < e^{ct} \right\rbrace \right) \times ]\underline{t}_\mu, +\infty[$. Let $j \in \mathbb{Z}^d$ be such that $x \in \Omega_{0} + a_j$. Since $\nu_2 < r_1$ (the radius of the uniform interior ball condition), we deduce the existence of  $z_x \in \Omega_{0, \frac{\nu_2}{2}}  $ such that 
\begin{equation}\label{zx} 
x \in B(z_x + a_j, \frac{\nu_2}{4}) \text{ and  } \forall y \in B(z_x + a_j, \frac{\nu_2}{4}) \text{ there holds } y \in  (\Omega_{0} + a_j)_{\frac{\nu}{4}} \cap \left\lbrace |y| < e^{ct} \right\rbrace. 
\end{equation} % We should note that $B(z_x, \frac{\delta}{2}) \subset \Omega_{0,\delta} \cap \left\lbrace |x| < e^{\frac{(|\lambda_0|-\delta)t}{1+2\alpha}} \right\rbrace$.  \\
Remarking that $\frac{n(x,t)}{n_+(x)} = \frac{n(x,t)}{n_+(x-a_j)}$, we are going to control each terms of the following decomposition: 
\begin{equation*}
 \frac{n(x,t)}{n_+(x)}  = \frac{n(x,t)}{n_{M,z_x}(x-a_j ,t-t_\sigma)} \times \frac{n_{M,z_x}(x-a_j ,t-t_\sigma)}{n_{M,+}(x-a_j )} \times \frac{n_{M,+}(x-a_j )}{n_+(x-a_j)} = \mathrm{I} \times \mathrm{II} \times \mathrm{III}
\end{equation*}
where $n_{M,z_x}$ is defined in \eqref{nM}.}

\textcolor{black}{\textbf{Control of $\mathrm{I}$.}\\
Thanks to \eqref{step1theorem4} and \eqref{zx}, it follows that
\[ \forall y \in B(z_x + a_j, \frac{\nu}{4}), \quad \sigma \leq n(y,t_\sigma).\]
Recalling that $n_{M,z_x}(x , 0) = \sigma 1_{B(z_x,\frac{\nu_2}{4})}(x)$, we conclude thanks to the comparison principle that 
\[\forall (y,s) \in \mathbb{R}^d \times [t_\sigma, +\infty[ \text{ we have } n_{M, z_x}(y-a_j, s-t_\sigma) \leq n(y,s) .\]
Since $t>t_\sigma$, we conclude that 
\begin{equation}\label{control_4}
1  \leq \frac{n(x,t)  }{n_{M,z_x}(x-a_j,t-t_\sigma)}.
\end{equation}}

\textcolor{black}{\textbf{Control of $\mathrm{II}$.}\\
Since $t-t_\sigma>t_\mu$, we deduce thanks to \eqref{t_mu} that
\begin{equation}\label{control_1}
\left(1-\mu\right)^{\frac{1}{2}} \leq \frac{n_{M,z_x}(x-a_j,t-t_\sigma)}{n_{M,+}(x-a_j)}.
\end{equation}}

\textcolor{black}{\textbf{Control of $\mathrm{III} $.}\\
Since $x-a_j \in \Omega_0$, we deduce thanks to \eqref{tmu2} that
\begin{equation}\label{control_2} 
\left(1-\mu\right)^{\frac{1}{2}} \leq \frac{n_{M,+}(x-a_j )}{n_+(x-a_j)}.
\end{equation}}

\textcolor{black}{Combining \eqref{control_1}, \eqref{control_2} and \eqref{control_4}, we conclude that for all $ (x,t) \in  \left(\Omega_{\nu_2} \cap \left\lbrace |x|< e^{ct} \right\rbrace \right) \times ]\underline{t}_\mu, +\infty[  $, we obtain
\begin{equation*}
1-\mu  \leq \frac{n(x,t)}{n_{M,z_x}(x-a_j ,t-t_\sigma)} \times \frac{n_{M,z_x}(x-a_j ,t-t_\sigma)}{n_{M,+}(x-a_j )} \times \frac{n_{M,+}(x-a_j )}{n_+(x-a_j)} = \frac{n(x,t)}{n_+(x-a_j)} = \frac{n(x,t)}{n_+(x)}.
\end{equation*}}

\textcolor{black}{This concludes the proof of Theorem \ref{theorem_convergence_1}.}
\end{proof}

\textcolor{black}{It remains to prove the claims \eqref{tmu2} and \eqref{t_mu}. The proof of \eqref{tmu2} relies on the uniqueness result stated in Theorem \eqref{theorem_existence_uniqueness_stationnary_state}.
\begin{proof}[Proof of \eqref{tmu2}]
The map $( M \in ]M_0, +\infty[ \mapsto n_{M,+})$ is increasing as $n_M$ is a sub-solution to the equation for $n_{M'}$ for $M'>M$. It converges to a weak solution of \eqref{equation_1_stat_1}. By fractional elliptic regularity, the limit is a strong solution of \eqref{equation_1_stat_1}. We conclude thanks to the uniqueness of the solution of \eqref{equation_1_stat_1} stated in Theorem \ref{theorem_existence_uniqueness_stationnary_state}. 
\end{proof}
The proof of \eqref{t_mu} relies on a compactness argument.
%\begin{proposition}
%For any $\mu>0$ and $M>\mathrm{diam}(\Omega_0)$, there exists $t_\mu>0$ such that for all $z \in \Omega_{0, \nu}$ we have 
%
%\end{proposition}
\begin{proof}[Proof of \eqref{t_mu}]
For a fixed $z \in \Omega_\nu$, the proof of convergence of $n_{M,z}$ to $n_{M,+}$ is classical thanks to \eqref{H1}. For each $z  \in \Omega_{0, \nu}$, there exists $t_z>0$ such that 
\[\forall (x,t) \in \mathbb{R}^d \times ] t_z, +\infty[ , \quad (1-\mu)^\frac{1}{2} \leq \frac{n_{M,z}(x,t)}{n_{M,+}(x)}.\]
We claim that $\underset{ z \in \Omega_{0, \nu}}{\sup} t_z < +\infty$. This assertion is true by compactness of $\overline{\Omega}_{0, \nu}$ (otherwise there exists $\overline{z} \in \overline{\Omega}_{0, \nu}$ such that $t_{\overline{z}}= +\infty$ which is a contradiction). 
\end{proof}}

\textbf{Aknowledgement}

S. Mirrahimi is grateful for partial funding from the European Research Council (ERC) under the European Union's Horizon 2020 research, innovation programme (grant agreement No639638),  held by Vincent Calvez and the chaire Mod\'{e}lisation Math\'{e}matique et Biodiversit\'{e} of V\'{e}olia Environment - Ecole Polytechnique - Museum National d'Histoire Naturelle - Fondation X. J.M. Roquejoffre was partially funded by the  ERCGrant Agreement n.  321186 - ReaDi - Reaction-Diffusion Equations, Propagation and Modelling held by Henri Berestycki, as well as the ANR project NONLOCAL ANR-14-CE25-0013.

\bibliographystyle{plain} % Le style est mis entre accolades.

%\bibliography{/home/aleculie/Documents/Bibliographie/bibliographie.bib} % mon fichier de base de données s'appelle bibli.bib
%\bibliography{C:/Users/Alexis/Desktop/Austin/Documents/Bibliographie/bibliographie}
{\footnotesize
\bibliography{/home/aleculie/Documents/Austin/Documents/Bibliographie/bibliographie}}
%\bibliography{/home/aleculie/Documents/Austin/Documents/Bibliographie/bibliographie.bib}

\end{document}